\documentclass{amsart}
\usepackage{amssymb,amstext,amsmath,amscd,amsthm,amsfonts,enumerate,graphicx,latexsym,stmaryrd,multicol}
\usepackage[usenames]{color}
\usepackage[all]{xy}

\newtheorem{thm}{Theorem}[section]
\newtheorem{lem}[thm]{Lemma}
\newtheorem{prop}[thm]{Proposition}
\newtheorem{cor}[thm]{Corollary}
\theoremstyle{definition}
\newtheorem{dfn}[thm]{Definition}
\newtheorem{ques}[thm]{Question}
\newtheorem{rem}[thm]{Remark}

\newtheorem{ex}[thm]{Example}

\newtheorem{emp}[thm]{}
\theoremstyle{remark}

\newtheorem*{ac}{Acknowledgments}

\def\cok{\operatorname{Cok}}
\def\depth{\operatorname{depth}}
\def\Ext{\operatorname{Ext}}

\def\dim{\operatorname{dim}}
\def\grade{\operatorname{grade}}
\def\H{\operatorname{H}}

\def\Hom{\operatorname{Hom}}
\def\id{\operatorname{id}}
\def\image{\operatorname{Im}}
\def\ker{\operatorname{Ker}}
\def\L{\mathbf{L}}
\def\lim{\operatorname{lim}}
\def\m{\mathfrak{m}}
\def\n{\mathfrak{n}}
\def\p{\mathfrak{p}}
\def\pd{\operatorname{pd}}
\def\q{\mathfrak{q}}

\def\RHom{\operatorname{\mathbf{R}Hom}}

\def\spec{\operatorname{Spec}}

\def\Tor{\operatorname{Tor}}

\def\fd{\operatorname{fd}}
\def\C{\mathcal{C}}
\def\G{\mathcal{G}}
\def\Gdim{\operatorname{G-dim}}
\def\Gid{\operatorname{Gid}}
\def\Gfd{\operatorname{Gfd}}
\def\Gpd{\operatorname{Gpd}}
\def\Hdim{\operatorname{H-dim}}
\def\Hid{\operatorname{Hid}}
\def\Rfd{\operatorname{Rfd}}

\def\CIdim{\operatorname{CI-dim}}
\begin{document}
\allowdisplaybreaks
\title{Finiteness of homological dimensions of Ext modules}
\author{Kaito Kimura}
\address{Graduate School of Mathematics, Nagoya University, Furocho, Chikusaku, Nagoya 464-8602, Japan}
\email{m21018b@math.nagoya-u.ac.jp}

\begin{abstract}
Let $R$ be a commutative Noetherian local ring and let $M$ and $N$ be nonzero finitely generated $R$-modules.
In this paper, we investigate how the finiteness of the homological dimension of Ext modules between $M$ and $N$ affects that of $M$ and $N$.
One of our main result states that if $\operatorname{Ext}^i_R(M,N)$ has finite projective dimension for any $0\le i\le \operatorname{Rfd}_R M$, where $\operatorname{Rfd}_R M$ is the (large) restricted flat dimension of $M$, then $M$ has finite projective or injective dimension if and only if $N$ does.
\end{abstract}
\maketitle
\section{Introduction}

Throughout the present paper, all rings are assumed to be commutative and Noetherian.
Let $R$ be a local ring and let $M$ and $N$ be nonzero finitely genrated $R$-modules.
The characterization of the homological dimensions of two modules in terms of those of the derived Hom complex or its (finitely many) homologies has been actively studied; see \cite{DHM, GP, GS, GT, HJM, MJ2, MJ, ZG} for instance.
Considering a suitable representation of the complex, one sees that the finiteness of the appropriate homological dimensions of the two modules implies that of the derived Hom complex, as follows:
\begin{align}
\tag{I} \pd_R M<\infty \ {\rm and} \ \pd_R N<\infty &\implies \pd_R (\RHom_R(M,N))<\infty, \\ 
\tag{II} \id_R M<\infty \ {\rm and} \ \id_R N<\infty &\implies \pd_R (\RHom_R(M,N))<\infty, \\ 
\tag{III} \pd_R M<\infty \ {\rm and} \ \id_R N<\infty &\implies \id_R (\RHom_R(M,N))<\infty, 
\end{align}
where, $\pd$ denotes the projective dimension and $\id$ denotes the injective dimension.
The formulas for the Poincar\'{e} and Bass series given by Foxby \cite{F} assert that converse of (III) holds.
In particular, when $\Ext^i_R(M,N)$ has finite injective dimension for any $0\le i\le \depth R-\depth M$, $M$ has finite projective dimension if and only if $N$ has finite injective dimension.

It is a natural question to ask what can be said for (I) and (II) in this context.
Holanda, Jorge-P\'{e}rez, and Mendoza-Rubio \cite{HJM} posed the following question:
when $\Ext^i_R(M,N)$ has finite projective dimension for any $0\le i\le \depth R-\depth M$, is it true that
\begin{align}
\tag{i} \pd_R M<\infty &\iff \pd_R N<\infty, \ {\rm and}  \\ 
\tag{ii} \id_R N<\infty &\iff \id_R M<\infty?
\end{align}
They proved that ``$\Rightarrow$'' holds for both (i) and (ii), and similarly showed the corresponding statements for Gorenstein homological dimensions. 
Moreover, in the Cohen--Macaulay case, they also resolved ``$\Leftarrow$'' of (i). 
However, ``$\Leftarrow$'' of (ii), and these questions outside the Cohen--Macaulay setting, remain open. 

The main results of this paper provide an affirmative answer to the above questions.
Here, $\Hdim$ denotes either the projective dimension or the Gorenstein projective dimension, and correspondingly, $\Hid$ denotes either the injective dimension or the Gorenstein injective dimension.
The (large) restricted flat dimension $\Rfd_R M$ is defined as ${\rm sup}\{\depth R_\p-\depth M_\p \mid \p\in\spec R\}$.

\begin{thm}{\rm [Corollary \ref{pd or Gdim cor 2}]}\label{pd or Gdim Main 1}
Let $R$ be a local ring, and $M$ and $N$ nonzero finitely generated $R$-modules with $\pd_R N<\infty$ and $\Hdim_R\Ext^i_R(M,N)<\infty$ for any $0\le i\le \Rfd_R M$.
Then $\Hdim_R M<\infty$.
\end{thm}

\begin{thm}{\rm [Corollary \ref{id or Gid cor 2}]}\label{id or Gid Main 1}
Let $R$ be a local ring, $M$ and $N$ nonzero finitely generated $R$-modules with $\id_R M<\infty$ and $\Hdim_R\Ext^i_R(M,N)<\infty$ for any $0\le i\le \depth R-\depth M$.
In the case $\Hdim=\Gdim$, assume that $R$ admits a dualizing complex.
Then $\Hid_R N<\infty$.
\end{thm}

It is worth noting that we have $\Rfd_R M=\depth R-\depth M$ when $M$ has finite (Gorenstein) projective dimension.
From this perspective, the above theorems essentially provide affirmative answers to the questions (i) and (ii).
Holanda, Jorge-P\'{e}rez, and Mendoza-Rubio discovered a certain compatibility between Ext modules and the tensor product with the canonical module.
Our main idea is to treat the situation where higher Ext modules have finite length as the base case, thereby avoiding the complexities that arise when replacing the canonical module with the dualizing complex.
There is also a result analogous to Theorem \ref{pd or Gdim Main 1} for complete intersection dimension; see Proposition \ref{C.I-dim cor}.

Our main results generalize several previous results.
When the ring is Cohen–Macaulay, $\Rfd_R M=\depth R-\depth M$ holds, so Theorem \ref{pd or Gdim Main 1} recovers \cite[Corollaries 5.7, 5.8, and 5.9]{HJM}.
Theorem \ref{pd or Gdim Main 1}, when applied with $N=R$ or assuming the vanishing of certain Ext modules, can be seen as a generalization of \cite[Proposition 13]{DHM}, \cite[Proposition 6.4]{HJM}, and \cite[Theorem 1.3]{MJ}; see Remark \ref{recoverd privious results}.

What should be of interest next are concrete examples of pairs of modules $M$ and $N$ that satisfy the assumptions of our theorems.
A typical pair of $M$ and $N$ for Theorem \ref{pd or Gdim Main 1} is given by a (G-)perfect module $M$, that is, a module whose (Gorenstein) projective dimension is equal to the grade of $M$, and $N=R$.
This paper provides several more nontrivial examples that are not (G-)perfect by considering extensions of (G-)perfect modules.

The organization of this paper is as follows. 
Section 2 is devoted to the conventions and to explaining the main motivation behind our research.
In Section 3, we provide proofs for all the results presented in this paper, including Theorems \ref{pd or Gdim Main 1} and \ref{id or Gid Main 1}.
In Section 4, we observe several interesting examples that lend support to our results and study complete intersection dimension a little.

\section{Notation and Motivation}

This section aims to introduce the conventions and explain the main motivation behind our research.
We begin by defining the notation that will be used throughout the paper.

\begin{dfn}
Let $R$ be a ring and $X=(\cdots \to X^n\xrightarrow{d^n} X^{n+1} \xrightarrow{d^{n+1}} \cdots)$ a complex of $R$-modules.
We denote isomorphisms by the symbol $\cong$ and quasi-isomorphisms by the symbol $\simeq$.
We omit subscripts/superscripts if there is no ambiguity.

(1) Let $M$ be an $R$-module and let $(\cdots\to F_2 \xrightarrow{f_2} F_1\xrightarrow{f_1} F_0 \to 0)$ be a projective resolution of $M$.
The {\em $n$th syzygy} $\Omega_R^n M$ is defined as $\cok f_{n+1}$.
(Note that syzygies are isomorphic up to projective summands,  independently of the choice of projective resolution.)
Denote by $E_R(M)$ the injective hull of $M$.
If $M$ is finitely generated over $R$, we define $\Rfd_R M={\rm sup}\{\depth R_\p-\depth M_\p \mid \p\in\spec R\}$.
By \cite[Theorem 1.1]{AIL}, $0\le \Rfd_R M<\infty$ if $M$ is nonzero.

(2) For each integer $m$, $X[m]$ denotes the complex $X$ shifted $m$ degrees, defined by $X[m]^n=X^{n+m}$.
A complex $X$ is called \textit{bounded to the left} if $X^n=0$ for all $n\ll 0$; similarly, $X$ is called \textit{bounded to the right} if $X^n=0$ for all $n\gg 0$.
A complex which is bounded to the left and to the right is said to be \textit{bounded}.
The category of complexes of $R$-modules is denoted by $\C(R)$.
Given a full subcategory $\mathcal{A}(R)$ of $\C(R)$, we write $\mathcal{A}_{\sqsubset}(R)$, $\mathcal{A}_{\sqsupset}(R)$, and $\mathcal{A}_{\square}(R)$ for the subcategories of $\mathcal{A}(R)$ consisting of complexes that are bounded to the left, bounded to the right, and bounded, respectively.

(3) Let $M$ be an $R$-module.
If there exists an exact complex $X$ of projective $R$-modules such that $M\cong \image d^0$ and the complex $\Hom_R(X,P)$ is exact for any projective $R$-module $P$, then $M$ is called \textit{Gorenstein projective}.
A finitely generated Gorenstein projective module is also called \textit{totally reflexive}.
If there exists an exact sequence $X$ of flat $R$-modules such that $M\cong \image d^0$ and the complex $I\otimes_R X$ is exact for any injective $R$-module $I$, then $M$ is said to be \textit{Gorenstein flat}.
If there exists an exact sequence $X$ of injective $R$-modules such that $M\cong \image d^0$ and the complex $\Hom_R(I,X)$ is exact for any injective $R$-module $I$, then $M$ is said to be \textit{Gorenstein injective}.
We denote by $\C^\mathrm{P}(R)$, $\C^\mathrm{F}(R)$, $\C^\mathrm{I}(R)$, $\G(R)$, $\C^\mathrm{GP}(R)$, $\C^\mathrm{GF}(R)$, and $\C^\mathrm{GI}(R)$ the full subcategories of $\C(R)$ consisting of complexes whose terms are projective, flat, injective, totally reflexive, Gorenstein projective, Gorenstein flat, and Gorenstein injective modules, respectively.
Projective, flat, and injective modules are also Gorenstein projective, Gorenstein flat, and Gorenstein injective, respectively.

(4) Let $Y$ and $Z$ be complexes of $R$-modules.
When $Y$ and $Z$ are quasi-isomorphic to complexes that are bounded to the right and to the left, respectively, we define the following homological dimensions:
\begin{itemize}
\item $\pd_R Y:={\rm inf}\{{\rm sup}\{n \mid X^{-n}\ne 0\} \mid Y\simeq X\in \C_{\sqsupset}^\mathrm{P}(R)\}$ is the \textit{projective dimension};
\item $\fd_R Y:={\rm inf}\{{\rm sup}\{n \mid X^{-n}\ne 0\} \mid Y\simeq X\in \C_{\sqsupset}^\mathrm{F}(R)\}$ is the \textit{flat dimension};
\item $\id_R Z:={\rm inf}\{{\rm sup}\{n \mid X^{n}\ne 0\} \mid Z\simeq X\in \C_{\sqsubset}^\mathrm{I}(R)\}$ is the \textit{injective dimension};
\item $\Gdim_R Y:={\rm inf}\{{\rm sup}\{n \mid X^{-n}\ne 0\} \mid Y\simeq X\in \G_{\sqsupset}(R)\}$ is the \textit{G-dimension};
\item $\Gpd_R Y:={\rm inf}\{{\rm sup}\{n \mid X^{-n}\ne 0\} \mid Y\simeq X\in \C_{\sqsupset}^\mathrm{GP}(R)\}$ is the \textit{Gorenstein projective dimension};
\item $\Gfd_R Y:={\rm inf}\{{\rm sup}\{n \mid X^{-n}\ne 0\} \mid Y\simeq X\in \C_{\sqsupset}^\mathrm{GF}(R)\}$ is the \textit{Gorenstein flat dimension};
\item $\Gid_R Z:={\rm inf}\{{\rm sup}\{n \mid X^{n}\ne 0\} \mid Z\simeq X\in \C_{\sqsubset}^\mathrm{GI}(R)\}$ is the \textit{Gorenstein injective dimension}.
\end{itemize}
When $Y$ and $Z$ are $R$-modules and each homological dimension is zero, it is a projective, flat, injective, totally reflexive, Gorenstein projective, Gorenstein flat, or Gorenstein injective module, respectively.

(5) In this paper, there are several assertions and arguments that hold for both the projective dimension and the G-dimension.
In such cases, we collectively denote $\pd$ and $\Gdim$ by ``$\Hdim$''.
Depending on whether $\Hdim$ is treated as the projective dimension or the G-dimension, we denote the injective dimension or the Gorenstein injective dimension by ``$\Hid$'', respectively.

(6) Suppose that $R$ is a homomorphic image of a finite-dimensional Gorenstein ring.
Then there exists a surjective ring homomorphism from $S$ to $R$, where $S$ is a Gorenstein ring of dimension $d=\dim R$.
Let $I=(0\to I^0\to \cdots \to I^d\to 0)$ be a minimal injective resolution of $S$-module $S$.
The \textit{dualizing complex} of $R$ is defined as $D:=\Hom_S(R,I)$.
It is easy to see that 
$$
D^i\cong \bigoplus_{\substack{\p\in\spec R \\ \dim R/\p=d-i}} E_R(R/\p), \\ 
$$
for all $0\le i\le d$.
(Unless otherwise specified, we assume that the dualizing complex takes this form.)
The module $\omega:=H^0(D)$ is called a \textit{canonical module} of $R$.
If $R$ is Cohen--Macaulay, we have $\omega\simeq D$.
\end{dfn}

We note the following well-known facts as a remark, and use them hereafter without further mention.

\begin{rem}\label{Well-known fact}
Let $R$ be a ring, and $W$ a complex of $R$-modules.

(1) Let $0\to X\to Y\to Z\to 0$ be a short exact sequence of complexes of $R$-modules.
If either $W$ or $Z$ is a complex of flat modules, then $0\to X\otimes_R W\to Y\otimes_R W\to Z\otimes_R W\to 0$ is exact; see \cite[Example 2.4.17 and Corollary 5.4.3]{CFH} for instance.

(2) Suppose that $R$ admits a dualizing complex $D$ and that $W$ is quasi-isomorphic to a bounded complex of $R$-modules.
It follows from \cite[1.3, Theorems 4.1 and 4.4]{CFrH} that $W$ has finite (Gorenstein) flat dimension if and only if $W\otimes_R^\L D$ has finite (Gorenstein) injective dimension, and that $W$ has finite (Gorenstein) injective dimension if and only if $\RHom_R(D,W)$ has finite (Gorenstein) projective dimension.
Also, as in \cite[Theorem 4.4]{CFrH}, since the finiteness of the Gorenstein flat dimension and that of the Gorenstein projective dimension are equivalent.
Moreover, if all homologies of $W$ are finitely generated, then $\pd_R W=\fd_R W$ and $\Gdim_R W=\Gpd_R W=\Gfd_R W$ hold.

(3) If $W$ has finite projective dimension or finite G-dimension, then so does $W_\p$ for any $\p\in\spec R$.
Moreover, if $W$ is a finitely generated $R$-module, then for each $\p\in\spec R$, the projective dimension and G-dimension of $W_\p$ and $W$ coincide with $\depth R_\p-\depth W_\p$ and $\Rfd_R W$, respectively; see \cite[Theorem 1.3.3]{BH} and \cite[Theorem (2.3.13)]{C}.
In the same case as in (2), by using the duality observed there, we see that the finiteness of the (Gorenstein) injective dimension is also preserved under localization.


\end{rem}

Holanda, Jorge-P\'{e}rez, and Mendoza-Rubio \cite{HJM} suggested the following questions for local rings.

\begin{ques}\cite[See Section 6]{HJM}\label{HJM Section 6}
Let $R$ be a local ring of depth $t$, and let $M$ and $N$ be nonzero finitely generated $R$-modules with $\pd_R (\Ext^i_R(M,N))<\infty$ for all $0\le i\le t$. 
Is the condition $\pd_R M<\infty$ equivalent to $\pd_R N<\infty$?
\end{ques}

\begin{ques}\cite[Question 6.1]{HJM}\label{HJM Q6.1}
Let $R$ be a local ring of depth $t$, and let $M$ and $N$ be nonzero finitely generated $R$-modules with $\pd_R N<\infty$. Do the following hold?
\begin{enumerate}[\rm(1)]
\item If $\Hdim_R (\Ext^i_R(M,N))<\infty$ for all $0\le i\le t$, then $\Hdim_R M<\infty$.
\item If $\Hdim_R (\Hom_R(M,N))<\infty$ and $\Ext^i_R(M,N)=0$ for all $1\le i\le t$, then $\Hdim_R M=0$.
\end{enumerate}
\end{ques}

\begin{ques}\cite[Question 6.9]{HJM}\label{HJM Q6.9}
Let $R$ be a local ring of depth $t$, and let $M$ and $N$ be nonzero finitely generated $R$-modules with $\pd_R (\Ext^i_R(M,N))<\infty$ for all $0\le i\le t$. 
Is the condition $\id_R M<\infty$ equivalent to $\id_R N<\infty$?
\end{ques}

One may also consider formulating the hypothesis of the question as $0\le i\le t-\depth M$ instead of $0\le i\le t$.
The reason they raised these questions is that they proved Question \ref{HJM Q6.1} in the affirmative when $R$ is Cohen--Macaulay.
Even for a general $R$, it is given in \cite{DHM, HJM} that Question \ref{HJM Q6.1} holds when $N=R$.
Furthermore, Holanda, Jorge-P\'{e}rez, and Mendoza-Rubio \cite{HJM}, focusing on the equivalence of the finiteness of the projective dimensions of $M$ and $N$ under the condition $\pd_R (\Ext^i_R(M,N))<\infty$ for all $0\le i\le t$, have resolved Question \ref{HJM Section 6} in several cases; see \cite[Corollary 5.9, Proposition 6.2]{HJM}.
Question \ref{HJM Q6.9} was posed dually to Question \ref{HJM Section 6}, in light of the quasi-isomorphism in \ref{underlying methods}(3). 
However, it is shown by \cite[Theorems 3.3 and 3.6]{HJM} that the difficulty of this equivalence resides in Question \ref{HJM Q6.1}(1) and in the following Question \ref{update question}(1).

\begin{ques}\label{update question}
Let $R$ be a local ring of depth $t$, and let $M$ and $N$ be nonzero finitely generated $R$-modules with $\id_R M<\infty$. Do the following hold?
\begin{enumerate}[\rm(1)]
\item If $\Hdim_R (\Ext^i_R(M,N))<\infty$ for all $0\le i\le t$, then $\Hid_R N<\infty$.
\item If $\Hdim_R (\Hom_R(M,N))<\infty$ and $\Ext^i_R(M,N)=0$ for all $1\le i\le t$, $\Hid_R N<\infty$ and $\depth M=t$.
\end{enumerate}
\end{ques}

Since a necessary and sufficient condition for $\Hdim_R M=0$ is that $\Hdim_R M<\infty$ and $\depth M=t$, and taking into account that depth is preserved under Foxby equivalence, it is seen that Question \ref{update question} is the dual version of Question \ref{HJM Q6.1}.
In the next section, we essentially resolve all of these questions (that is, by slightly adjusting the degrees of the assumed Ext modules).



\section{Main results}

In this section, we investigate the equivalence of the finiteness of homological dimensions for pairs of finitely generated modules and their Ext modules, and prove the main results of this paper.
To prove the main results of this paper, we prepare several basic lemmas.

\begin{lem}\label{fpd dualizing ext vanishing}
Let $(R,\m)$ be a local ring, $M$ and $N$ finitely generated $R$-modules, and $D$ a dualizing complex of $R$.
\begin{enumerate}[\rm(1)]
\item If $\pd_R N<\infty$, then $H^i(\RHom_R(M, N\otimes_R^\L D))=0$ for all integers $i>\dim R-\depth M$.
\item If $\id_R M<\infty$, then $H^i(\RHom_R(\RHom_R(D,M), N))=0$ for all integers $i>\depth R-\depth M$.
\item If $\Gdim_R M<\infty$ and $\pd_R N<\infty$, then $\Ext_R^i(M,N)=0$ for all integers $i>\depth R-\depth M$.
\item If $\id_R M<\infty$ and $\Gid_R N<\infty$, then $\Ext_R^i(M,N)=0$ for all integers $i>\depth R-\depth M$.
\end{enumerate}
\end{lem}

\begin{proof}
We prove (1) and (3).
We use induction on $\pd_R N$.
If $\pd_R N=0$, the claim is obvious; see \cite[(1.2.7)]{C}.
Suppose $\pd_R N>0$. 
We consider a short exact sequence $0\to \Omega N\to F\to N\to 0$, where $F$ is a finitely generated free $R$-module.
This implies exact sequences 
\begin{align*}
H^i(\RHom_R(M, F\otimes_R^\L D))& \to H^i(\RHom_R(M, N\otimes_R^\L D))\to H^{i+1}(\RHom_R(M, \Omega N\otimes_R^\L D)) \ {\rm and}\  \\
&\Ext_R^i(M, F) \to \Ext_R^i(M, N) \to \Ext_R^{i+1}(M, \Omega N)
\end{align*}
for any integer $i$.
The claim follows from the induction hypothesis.

We show (2) and (4).
We may assume that $M$ is nonzero and hence $R$ is Cohen--Macaulay; see \cite[Corollary 9.6.2, Remark 9.6.4(a)]{BH}.
Let $\omega$ be the canonical module of $R$.
We use induction on $r=\depth R-\depth M$.
If $r=0$, the claim is obvious since $M$ is isomorphic to a finite direct sum of copies of $\omega$; see \cite[(3.4.9) and (6.2.5)]{C} for instance.
Suppose $r>0$. 
Take a short exact sequence $0\to K\to L \to M\to 0$, where $L$ is isomorphic to a finite direct sum of copies of $\omega$; see \cite[Exercise 3.3.28]{BH}.
Note that $\dim R-\depth K=r-1$.
The above short exact sequence implies exact sequences 
\begin{align*}
H^{i-1}(\RHom_R(\RHom_R(\omega,K), N)) &\! \to\! H^i(\RHom_R(\RHom_R(\omega,M), N))\!\to\! H^i(\RHom_R(\RHom_R(\omega,L), N)) \\
 \ {\rm and}\ &\Ext_R^{i-1}(K,N) \to \Ext_R^i(M, N) \to \Ext_R^i(L, N)
\end{align*}
for any integer $i$.
The claim follows from the induction hypothesis.
\end{proof}

\begin{lem}\label{flat cpx bounded}
Let $R$ be a ring, $X$ a complex of flat $R$-modules, and $\Phi: Y\to Z$ a quasi-isomorphism of bounded complexes of $R$-modules.
Then $\Phi\otimes_R X: Y\otimes_R X\to Z\otimes_R X$ is a quasi-isomorphism.
\end{lem}

\begin{proof}
We prove that $H^m(\Phi\otimes_R X):H^m(Y\otimes_R X)\to H^m(Z\otimes_R X)$ is an isomorphism for any integer $m$.
Since $Y$ and $Z$ are bounded, there exists an integer $n\ge 0$ such that $Y^i=0=Z^i$ whenever $i<-n$ or $n<i$.
For any integer $m$, $H^m(Y\otimes_R X)$ is the cohomology of a complex of the following form:
$$
\bigoplus_{i=-n}^n Y^i\otimes_R X^{m-1-i} \to \bigoplus_{i=-n}^n Y^i\otimes_R X^{m-i} \to \bigoplus_{i=-n}^n Y^i\otimes_R X^{m+1-i}.
$$
The same holds for $H^m(Z\otimes_R X)$.
Putting $W:=(0\to X^{m-n-1}\to X^{m-n} \to \cdots \to X^{m+n}\to X^{m+n+1}\to 0)$,  it follows that $H^m(\Phi\otimes_R X)=H^m(\Phi\otimes_R W)$.
Since $W$ is a bounded complex of flat $R$-modules, $\Phi\otimes_R W$ is a quasi-isomorphism; see \cite[(A.4.1)]{C} for instance.
\end{proof}

The following theorem is one of the most important results in this paper.
The approach to the proof takes inspiration from the techniques in \cite[Corollary 3.7]{HJM}.
However, when $R$ is not Cohen--Macaulay, it is necessary to replace the canonical module with the dualizing complex, making it difficult to compute each Ext module exactly.
Our main idea is to avoid this complexity by considering the case where the homology has finite length as the base case.

\begin{thm}\label{fpd homology}
Let $(R,\m)$ be a local ring, $M$ and $N$ nonzero finitely generated $R$-modules with $\pd_R N<\infty$, and $r\ge\depth R-\depth M$ a nonnegative integer. 
Suppose that $\Hdim_R\Ext^i_R(M,N)<\infty$ for all $0\le i\le r$, and $\Ext^j_R(M,N)$ has finite length for any $j>r$. 
Then $\Hdim_R M<\infty$.
\end{thm}

\begin{proof}
We may assume $R$ is complete.
Set $d=\dim R$ and $t=\depth R$.
Let $D$ be a dualizing complex of $R$, and let $F=(\cdots \to F_1\to F_0\to 0)$ and $G=(0\to G_n \to \cdots \to G_0\to 0)$ be minimal free resolutions of $M$ and $N$.
Note that $\Hom_R(F, G)=: H=(0\to H^{-n}\to H^{-n+1} \to \cdots)$ is a complex of finitely generated free $R$-modules.
By \cite[(A.2.10)]{C} or \cite[Theorem 4.5.10]{CFH}, there is a quasi-isomorphism 
$$
\RHom_R(M, N\otimes_R^\L D)\simeq \Hom_R(F, G\otimes_R D)\cong \Hom_R(F, G)\otimes_R D=H\otimes_R D.
$$
Noting that $N\otimes_R^\L D\ne 0$ and it has finite injective dimension, if $\Hid_R (H\otimes_R D)<\infty$, we get $\Hdim_R M<\infty$; see \cite[Lemma 6.2.10]{CF} and the proof of \cite[Theorem 3.6(2)]{HJM}.
We prove $\Hid_R (H\otimes_R D)<\infty$.

Take a homomorphism $\Phi: P\to H$ of complexes such that $H^i(\Phi):H^i(P) \to H^i(H)\cong\Ext^i_R(M,N)$ is an isomorphism for every $i\le r$, where $P=(\cdots\to P^{r-1}\to P^r\to 0)$ is a complex of finitely generated free $R$-modules.
Let $C$ be the mapping cone of $\Phi$.
The short exact sequence
\[
  \xymatrix@C=15pt@R=10pt
  {
    H= (\cdots  \ar[r] 
    & H^{i-1} \ar[r] \ar[d]
    & H^i \ar[r] \ar[d]
    & \cdots ) \\
    C= (\cdots  \ar[r] 
    & C^{i-1}=H^{i-1}\oplus P^{i} \ar[r] \ar[d]
    & C^i=H^i\oplus P^{i+1} \ar[r] \ar[d]
    & \cdots ) \\
    P[1]= (\cdots \ar[r]
    & P^{i} \ar[r] 
    & P^{i+1} \ar[r] 
    & \cdots ) \\
  }
\]
of complexes of finitely generated free $R$-modules induces a short exact sequence $0\to H\otimes_R D\to C\otimes_R D\to (P\otimes_R D)[1]\to 0$.
By assumption, we obtain $\Hdim_R P<\infty$ and thus $\Hid_R (P\otimes_R D)<\infty$.
If $C\otimes_R D\simeq 0$, then $H\otimes_R D$ is quasi-isomorphic to $P\otimes_R D$ and we reach $\Hid_R (H\otimes_R D)<\infty$.
In the remainder of the proof, we show that $H^i(C\otimes_R D)=0$ for all integers $i$.

As $H^j(D)=0=H^k(P)$ for all $j>d-t$ and $k>r$, we have $H^i(P\otimes_R D)=H^i(P\otimes_R^\L D)=0$ for any $i>d-t+r$.
It follows from Lemma \ref{fpd dualizing ext vanishing}(1) that $H^i(H\otimes_R D)\cong H^i(\RHom_R(M, N\otimes_R^\L D))=0$ for every $i>d-t+r$, which means that $H^i(C\otimes_R D)=0$ for any $i>d-t+r$.

We denote the differential map $C^i\to C^{i+1}$ by $f^i$.
By the construction of $\Phi$ and $C$, if $i>r$, then $H^i(C)\cong H^i(H)\cong \Ext^i_R(M,N)$ and it has finite length; otherwise $H^i(C)=0$.
For any non-maximal prime ideal $\p$ of $R$, $C_\p$ is exact.
Since the complex $H$ is bounded to the left, for $i\ll 0$, $f^i$ is equal to the differential map $P^{i+1}\to P^{i+2}$, and thus $\cok f^i$ is totally reflexive.
By this and \cite[Theorem (1.4.8)]{C}, for every integer $i$, $\cok f^i_\p$ has finite G-dimension and thus totally reflexive, which implies that the $R_\p$-dual of $C_\p$ is exact.
It follows from \cite[Lemma (5.1.10)]{C} that $C\otimes_R E(R/\p)=C_\p\otimes_{R_\p} E(R/\p)\simeq 0$.
We obtain $C\otimes_R D^i\simeq 0$ for $i\ne d$.
For all $0\le i\le d-1$, the natural short exact sequence 
\[
  \xymatrix@C=15pt@R=10pt
  {
D(i+1):=(0  \ar[r]
    & 0 \ar[r] \ar[d]
    & D^{i+1} \ar[r] \ar@{=}[d]
    & \cdots \ar[r]
    & D^{d-1} \ar[r] \ar@{=}[d]
    & D^d \ar[r] \ar@{=}[d]
    & 0 ) \\
D(i):= (0   \ar[r] 
    & D^i \ar[r] \ar@{=}[d]
    & D^{i+1} \ar[r] \ar[d]
    & \cdots \ar[r] 
    & D^{d-1} \ar[r] \ar[d]
    & D^d \ar[r] \ar[d]
    & 0 ) \\
(0  \ar[r] 
    & D^i \ar[r] 
    & 0 \ar[r]
    & \cdots \ar[r]
    & 0 \ar[r] 
    & 0 \ar[r]
    & 0 )  \\
  }
\]
implies $C\otimes_R D=C\otimes_R D(0)\simeq C\otimes_R D(1)\simeq \cdots \simeq C\otimes_R D(d)$, and hence $H^i(C\otimes_R D)\cong H^{i-d}(C\otimes_R D^d)$ for any integer $i$.
Since $H^i(C)=0$ for any $i\le r$, $(\cdots \to C^{r-1} \to C^{r}\to C^{r+1}\to 0)$ is a free resolutin of $\cok f^r$.
For $i\ll 0$, $\cok f^i$ is totally reflexive, and thus $\cok f^r$ has finite G-dimension.
It follows from Lemma \ref{fpd dualizing ext vanishing}(3) that $\Ext_R^i(\cok f^r, R)=0$ for every $i>t$.
We see that $(\cdots \to C^{r-t-1}\otimes_R D^d \to C^{r-t}\otimes_R D^d \to C^{r-t+1}\otimes_R D^d)$ is exact since its Matlis dual is exact.
Thus $H^i(C\otimes_R D)\cong H^{i-d}(C\otimes_R D^d)=0$ for all $i\le d-t+r$.
\end{proof}

The following corollary is the most general statement of the results we have obtained.
Refer to Remark \ref{why localize} for the reason why conditions (1) and (2) need to be considered for each maximal ideal.

\begin{cor}\label{pd or Gdim cor}
Let $R$ be a ring, and let $M$ and $N$ be finitely generated $R$-modules.
Then $\Hdim_R M<\infty$ if for every maximal ideal $\m$ of $R$, at least one of the following holds:
\begin{enumerate}[\rm(1)]
\item $\Hdim_{R_\m} M_{\m}<\infty$
\item $N_\m\ne 0$, $\pd_{R_\m} N_\m<\infty$, and $\Hdim_{R_\m} \Ext^i_{R_\m}(M_\m,N_\m)<\infty$ for any $0\le i\le\Rfd_R M$.
\end{enumerate}
\end{cor}

\begin{proof}
Assume that $\Hdim_{R_\m} M_\m=\infty$ for some maximal ideal $\m$ of $R$.
By assumption, $\m$ belongs to $S:=\{\p\in\spec R\mid \Hdim_{R_\p} M_\p=\infty, N_\p\ne 0, \pd_{R_\p} N_\p<\infty, \Hdim_{R_\p} \Ext^i_{R_\p}(M_\p,N_\p)<\infty$ for any $0\le i\le\Rfd_R M\}$, so we can take a minimal element $\p$ of $S$.
Let $\q$ be a prime ideal of $R$ with $\q\subsetneq \p$.
Since $\pd_{R_\p} N_\p<\infty$ and $\Hdim_{R_\p} \Ext^i_{R_\p}(M_\p,N_\p)<\infty$ for any $0\le i\le\Rfd_R M$, we have $\pd_{R_\q} N_\q<\infty$ and $\Hdim_{R_\q} \Ext^i_{R_\q}(M_\q,N_\q)<\infty$ for any $0\le i\le\Rfd_R M$.
However, $\q$ is not in $S$ by the minimality of $\p$ in $S$.
Thus either $\Hdim_{R_\q} M_\q<\infty$ or $N_\q=0$ holds.
Lemma \ref{fpd dualizing ext vanishing}(3) forces $\Ext^j_{R_\q}(M_\q,N_\q)=0$ for any $j>\Rfd_R M$.
Therefore $\Ext^j_{R_\p}(M_\p,N_\p)$ has finite length for any $j>\Rfd_R M$.
Theorem \ref{fpd homology} concludes $\Hdim_{R_\p} M_\p<\infty$.
This contradicts $\p\in S$.
It follows from the above that $\Hdim_{R_\m} M_\m<\infty$ for any maximal ideal $\m$ of $R$, which induces $\Hdim_R M<\infty$.
\end{proof}

\begin{rem}\label{why localize}
The assumptions of Corollary \ref{pd or Gdim cor} concerning localized modules are seen to be essential in the following sense.
If $R$ is not a local ring, then for two finitely generated modules $M$ and $N$, it is possible that $\Ext^i_R(M, N)=0$ for all $i$.
This does not depend on the homological dimensions of $M$ and $N$.
For example, let $S$ be a ring, $\m$ and $\n$ distinct maximal ideals of $S$, $R=S/\m^2\n$, $M=S/\m$ and $N=S/\n$.
Considering the localization at every maximal ideal, we obtain $\pd_R N<\infty$, and $\Ext^i_R(M, N)=0$ for all $i$.
On the other hand, if $R_\m$ is not field (i.e. $\m S_\m\ne 0$), then $\pd_R M=\infty$ holds.
\end{rem}

We now list the corollaries that follow from the above results.
We begin with the case of a local ring, which is the most important case in this line of research.
The following corollary provides an affirmative answer to Question \ref{HJM Q6.1}(1) under essential assumptions, and is an immediate consequence of Corollary \ref{pd or Gdim cor}.
Indeed, when $\Hdim_R M<\infty$, we have $\depth R-\depth M=\Hdim_R M=\Rfd_R M$, so the hypothesis of Corollary \ref{pd or Gdim cor 2} is weaker than that of Question \ref{HJM Q6.1}(1).

\begin{cor}\label{pd or Gdim cor 2}
Let $R$ be a local ring, and $M$ and $N$ nonzero finitely generated $R$-modules with $\pd_R N<\infty$ and $\Hdim_R\Ext^i_R(M,N)<\infty$ for any $0\le i\le \Rfd_R M$.
Then $\Hdim_R M<\infty$.
\end{cor}

Lemma \ref{pd G-dim keisan} shows that the validity of Question \ref{HJM Q6.1}(1) implies that of \ref{HJM Q6.1}(2).
In particular, from Corollary \ref{pd or Gdim cor 2}, we immediately obtain Corollary \ref{pd or Gdim cor 2-2}, which ``essentially'' encompasses Question \ref{HJM Q6.1}(2).

\begin{lem}\label{pd G-dim keisan}
Let $R$ be a local ring, and $M$ and $N$ nonzero finitely generated $R$-modules.
\begin{enumerate}[\rm(1)]
\item If $\pd_R M<\infty$, then then $\pd_RM={\rm sup}\{i \mid\Ext_R^i(M,N)\ne0\}$.
\item If $\Gdim_R M<\infty$ and $\pd_R N<\infty$, then $\Gdim_RM={\rm sup}\{i \mid\Ext_R^i(M,N)\ne0\}$.
\end{enumerate}
\end{lem}

\begin{proof}
The assertion (1) follows by \cite[Page 154, Lemma 1(iii)]{Mat}.
To prove (2), we take a minimal free presentation $R^{\oplus m} \xrightarrow{\phi} R^{\oplus n}\to N \to 0$, where $\phi$ is an $n\times m$-matrix whose entries belong to the maximal ideal.
It is enough to show that $\Ext^r_R(M, N)\ne 0$; see Lemma \ref{fpd dualizing ext vanishing}(3).
Consider the short exact sequences $0\to \ker\phi \to R^{\oplus m} \to \image \phi \to 0$ and $0\to \image \phi \to R^{\oplus n} \to N \to 0$ of $R$-modules of finite projective dimension.
By Lemma \ref{fpd dualizing ext vanishing}(3), we have exact sequences
$$
\Ext^r_R(M, R^{\oplus m}) \to \Ext^r_R(M, \image \phi)\to 0 \ {\rm and} \ \Ext^r_R(M, \image \phi) \to \Ext^r_R(M, R^{\oplus n}) \to \Ext^r_R(M, N)\to 0
$$
where $r=\Gdim_RM$.
So $\Ext^r_R(M, R)^{\oplus m} \xrightarrow{\phi} \Ext^r_R(M, R)^{\oplus n} \to \Ext^r_R(M, N)\to 0$ is exact.
It follows from \cite[Theorem (1.2.7)]{C} that $\Ext^r_R(M, R)\ne 0$.
The Nakayama's lemma yields $\Ext^r_R(M, N)\ne 0$ because all entries of $\phi$ belong to the maximal ideal.
\end{proof}

\begin{cor}\label{pd or Gdim cor 2-2}
Let $R$ be a local ring, $M$ and $N$ nonzero finitely generated $R$-modules with $\pd_R N<\infty$, and $0\le r \le \Rfd_R M$.
Suppose that $\Hdim_R\Ext^i_R(M,N)<\infty$ for any $0\le i\le r$, and that $\Ext^j_R(M,N)=0$ for any $r< j\le \Rfd_R M$.
Then $\Hdim_R M=r=\Rfd_R M=\depth R-\depth M$.
\end{cor}

Question \ref{HJM Q6.1} had already been resolved in the Cohen--Macaulay case by
\cite[Corollaries 5.7 and 5.8]{HJM}.
We recover their results by combining our results with the following lemma.

\begin{lem}\label{Rfd CM lemma}
Let $R$ be a Cohen--Macaulay local ring, and $M$ a finitely generated $R$-module.
Then $\depth R_\p-\depth M_\p\le \depth R-\depth M$ for all $\p\in\spec R$, and hence $\Rfd_R M=\depth R-\depth M$.
\end{lem}

\begin{proof}
The claim follows from \cite[(A.6.2)]{C} as $\dim(R/\p)=\dim R-\dim R_\p=\depth R-\depth R_\p$.
\end{proof}

\begin{cor}\label{Rfd CM 1}
Let $R$ be a Cohen--Macaulay local ring, $M$ and $N$ nonzero finitely generated $R$-modules with $\pd_R N<\infty$ and $\Hdim_R\Ext^i_R(M,N)<\infty$ for any $0\le i\le \depth R-\depth M$. Then $\Hdim_R M<\infty$.
In addition, if $\Ext^i_R(M,N)=0$ for any $1\le i\le \depth R-\depth M$, then $\Hdim_R M=0$.
\end{cor}

\begin{rem}\label{recoverd privious results}
Combining the results obtained so far, \cite[(2) and (3) of Theorem 1.3]{MJ} are recovered.
In fact, if $R$ is Cohen--Macaulay or $M$ is locally totally reflexive on the punctured spectrum, then $\Rfd_R M={\rm max}\{\depth R-\depth M, 0\}$.
Therefore, by combining Corollary \ref{pd or Gdim cor 2} and Lemma \ref{pd G-dim keisan}(2), it is easy to see that a generalized statement of \cite[(2) and (3) of Theorem 1.3]{MJ}, where ``$R$'' is replaced by ``a module $N$ of finite projective dimension'', holds true.
We also obtain similar generalizations of \cite[Proposition 13]{DHM} and \cite[Proposition 6.4]{HJM} by imposing assumptions on more Ext modules.
The remaining problem is whether Corollary \ref{pd or Gdim cor 2} can be proved by replacing $\Rfd_R M$ with $\depth R-\depth M$.
\end{rem}

A case that is often considered, in addition to the local case, is the graded case as follows.
In fact, the examples in the next section, which are quotients of polynomial rings over a field, fall into these cases.
A graded ring $R=\bigoplus_{n\ge 0} R_n$ is called \textit{*local} if it has a unique \textit{*maximal ideal}, that is, a graded ideal that is properly contained in no graded ideal other than $R$.
For a finitely generated graded $R$-module $M$, we denote by $\depth M$ the maximal length of an $M$-regular sequence contained in a *maximal ideal.
In this case as well, the same statements as in the local case hold.

\begin{cor}\label{pd or Gdim cor 3}
Let $R$ be a *local ring, and $M$ and $N$ nonzero finitely generated graded $R$-modules with $\pd_R N<\infty$ and $\Hdim_R\Ext^i_R(M,N)<\infty$ for any $0\le i\le \Rfd_R M$.
Then $\Hdim_R M<\infty$.
\end{cor}

\begin{proof}
Let $\m$ be the *maximal ideal ideal of $R$.
Since $\m$ is a maximal ideal ideal of $R$, applying Corollary \ref{pd or Gdim cor 2} to $R_\m$, we obtain $\Hdim_{R_\m} M_\m<\infty$, which implies $\Hdim_R M<\infty$.
In fact, when $\Hdim=\pd$, this follows immediately from \cite[Proposition 1.5.15(e)]{BH}.
On the other hand, \cite[Proposition 1.5.15(c)]{BH} shows that the graded $R$-module $\Omega_R^r M$ is totally reflexive if and only if $(\Omega_R^r M)_\m$ is totally reflexive over $R_\m$, where $r=\Rfd_R M$.
\end{proof}

The following results can also be obtained by reducing it to Corollaries \ref{pd or Gdim cor 2-2} and \ref{Rfd CM 2}, in the same way as the proof of Corollary \ref{pd or Gdim cor 3}.
Thanks to \cite[Theorem 1.5.9]{BH}, $\Rfd$ is determined solely by the information at the localizations at homogeneous prime ideals.

\begin{cor}\label{pd or Gdim cor 3-2}
Let $R$ be a *local ring, $M$ and $N$ nonzero finitely generated graded $R$-modules with $\pd_R N<\infty$, and $0\le r \le \Rfd_R M$.
Suppose that $\Hdim_R\Ext^i_R(M,N)<\infty$ for any $0\le i\le r$, and that $\Ext^j_R(M,N)=0$ for any $r< j\le \Rfd_R M$.
Then $\Hdim_R M=r=\Rfd_R M=\depth R-\depth M$.
\end{cor}

\begin{cor}\label{Rfd CM 2}
Let $R$ be a Cohen--Macaulay *local ring, and $M$ and $N$ nonzero finitely generated graded $R$-modules with $\pd_R N<\infty$ and $\Hdim_R\Ext^i_R(M,N)<\infty$ for any $0\le i\le \depth R-\depth M$.
Then $\Hdim_R M<\infty$.
In addition, if $\Ext^i_R(M,N)=0$ for any $1\le i\le \depth R-\depth M$, then $\Hdim_R M=0$.
\end{cor}

From this point onward, we consider, dually to the preceding discussion, the finiteness of the (Gorenstein) injective dimension of modules.
The following is the (Gorenstein) injective dimension analogue of Theorem \ref{fpd homology} and is among the main results of this paper.
Our approach, which closely examines the structure of the dualizing complex, remains extremely effective even when the ring is Cohen--Macaulay.

\begin{thm}\label{fid homology}
Let $(R,\m)$ be a local ring, $M$ and $N$ nonzero finitely generated $R$-modules with $\id_R M<\infty$, and $r\ge\depth R-\depth M$ a nonnegative integer. 
Suppose that $\Hdim_R\Ext^i_R(M,N)<\infty$ for all $0\le i\le r$, and $\Ext^j_R(M,N)$ has finite length for any $j>r$.
In the case $\Hdim=\Gdim$, assume further that $R$ has dualizing complex.
Then $\Hid_R N<\infty$.
\end{thm}

\begin{proof}
When $\Hdim=\pd$, we can replace $R$ with its completion, so we may assume that $R$ admits a dualizing complex $D$ in order to prove theorem.
Put $d=\dim R$.
Since $M$ is a nonzero finitely generated $R$-module with finite injective dimension, $R$ is Cohen--Macaulay; see \cite[Corollary 9.6.2, Remark 9.6.4(a)]{BH}.
Let $\omega$ be the canonical module of $R$, and let $I=(0\to I^0\to \cdots\to I^d \to 0)$ and $J=(0\to J^0\to J^1\to \cdots)$ be minimal injective resolutions of $M$ and $N$.
Note that $\Hom_R(I,J)=:K=(0\to K^{-d} \to K^{-d+1}\to \cdots)$ is a complex of flat $R$-modules. 
By \cite[Theorem 4.5.13]{CFH}, we have a quasi-isomorphism 
$$
\RHom_R(\RHom_R(\omega, M), N)=\Hom_R(\Hom_R(\omega, I), J)\cong
\omega\otimes_R \Hom_R(I,J)=\omega\otimes_R K.
$$
Noting that $\RHom_R(\omega, M)\ne 0$ and it has finite projective dimension, if $\Hid_R (\omega\otimes_R K)<\infty$, we get $\Hid_R N<\infty$; see \cite[Lemma 6.2.10]{CF} and \cite[Theorem 2.9(1)]{HJM}. 
We prove $\Hid_R (\omega\otimes_R K)<\infty$.

Take a homomorphism $\Phi: P\to K$ of complexes such that $H^i(\Phi):H^i(P) \to H^i(K)\cong\Ext^i_R(M,N)$ is an isomorphism for every $i\le r$, where $P=(\cdots\to P^{r-1}\to P^r\to 0)$ is a complex of finitely generated free $R$-modules.
Let $C$ be the mapping cone of $\Phi$.
The short exact sequence
\[
  \xymatrix@C=15pt@R=10pt
  {
    K= (\cdots  \ar[r] 
    & K^{i-1} \ar[r] \ar[d]
    & K^i \ar[r] \ar[d]
    & \cdots ) \\
    C= (\cdots  \ar[r] 
    & C^{i-1}=K^{i-1}\oplus P^{i} \ar[r] \ar[d]
    & C^i=K^i\oplus P^{i+1} \ar[r] \ar[d]
    & \cdots ) \\
    P[1]= (\cdots \ar[r]
    & P^{i} \ar[r] 
    & P^{i+1} \ar[r] 
    & \cdots ) \\
  }
\]
of complexes of flat $R$-modules induces a short exact sequence $0\to \omega\otimes_R K\to \omega\otimes_R C\to (\omega\otimes_R P)[1]\to 0$.
By assumption, we obtain $\Hdim_R P<\infty$ and thus $\Hid_R (\omega\otimes_R P)<\infty$.
If $\omega\otimes_R C\simeq 0$, then $\omega\otimes_R K$ is quasi-isomorphic to $\omega\otimes_R P$.
In the remainder of the proof, we show that $H^i(\omega\otimes_R C)=0$ for all $i$.

It follows from Lemma \ref{fpd dualizing ext vanishing}(2) that $H^i(\omega\otimes_R K)\cong H^i(\RHom_R(\RHom_R(\omega, M), N))=0=H^i(\omega\otimes_R P)$ for every $i>r$, which means that $H^i(\omega\otimes_R C)=0$ for any $i>r$.

Lemma \ref{flat cpx bounded} deduces a quasi-isomorphism $\omega \otimes_R C\simeq D\otimes_R C$.
We denote the differential map $C^i\to C^{i+1}$ by $g^i$.
By the construction of $\Phi$ and $C$, if $i>r$, then $H^i(C)\cong H^i(K)\cong \Ext^i_R(M,N)$ and it has finite length; otherwise $H^i(C)=0$.
For any non-maximal prime ideal $\p$ of $R$, $C_\p$ is exact.
For $i\ll 0$, $g^i$ is equal to the differential map $P^{i+1}\to P^{i+2}$, and thus $\cok g^i$ is totally reflexive.
By this, $\cok g^j_\p$ has finite Gorenstein flat dimension for every integer $j$.
It follows from \cite[(5.2.15)]{C} that $\Gfd_{R_\p}(\cok g^{j+d}_\p) \le d$ and $0=\Tor_{d+1}^{R_\p}(E(R/\p), \cok g^{j+d}_\p)=\Tor_1^{R_\p}(E(R/\p), \cok g^{j}_\p)$ for any integer $j$, which means that $E(R/\p)\otimes_R C=E(R/\p)\otimes_{R_\p} C_\p\simeq 0$.
An argument similar to the proof of Theorem \ref{fpd homology} shows that $H^i(\omega \otimes_R C)\cong H^i(D\otimes_R C)\cong H^{i-d}(D^d\otimes_R C)$ for any integer $i$.

For the same reason as in the previous paragraph, we see that $\cok g^r$ has finite Gorenstein flat dimension at most $d$ since $H^i(C)=0$ for $i\le r$.
Note that $(\cdots\to C^r\to C^{r+1}\to 0)$ is a flat resolution of $\cok g^r$.
Again by \cite[(5.2.15)]{C}, we get $H^i(\omega \otimes_R C)\cong H^{i-d}(D^d\otimes_R C)\cong \Tor_{r+1-i+d}^R(D^d, \cok g^r)=0$ for $i \le r$.
The proof is now complete.
\end{proof}

The corollary below corresponds to Corollary \ref{pd or Gdim cor}.
However, a module whose localizations at all prime ideals have finite (Gorenstein) dimension does not necessarily have finite (Gorenstein) dimension itself.
This is due to the fact that the (Gorenstein) dimension of a finitely generated module over a local ring coincides with the depth of the ring; see \cite[Theorem 3.1.17]{BH} and \cite[Theorem 6.2.15]{C}.
For example, the injective dimension of an infinite-dimensional Gorenstein ring is infinite.

\begin{cor}\label{id or Gid cor}
Let $R$ be a ring, and let $M$ and $N$ be finitely generated $R$-modules.
In the case $\Hdim=\Gdim$, assume that $R_\p$ admits a dualizing complex for any $\p\in\spec R$.
Then $\Hid_{R_\p} N_\p<\infty$ for any prime ideal $\p$ of $R$ if at least one of the following holds for every maximal ideal $\m$ of $R$:
\begin{enumerate}[\rm(1)]
\item $\Hid_{R_\m} N_{\m}<\infty$
\item $M_\m\ne 0$, $\id_{R_\m} M_\m<\infty$, and $\Hdim_{R_\m} \Ext^i_{R_\m}(M_\m,N_\m)<\infty$ for any $0\le i\le\Rfd_R M$.
\end{enumerate}
In addition, if $R$ is finite-dimensional, and it admits a dualizing complex when $\Hdim=\Gdim$, then $\Hid_R N<\infty$.
\end{cor}

\begin{proof}
Assume that $\Hid_{R_\m} N_\m=\infty$ for some maximal ideal $\m$ of $R$.
By assumption, $\m$ belongs to $S:=\{\p\in\spec R\mid \Hid_{R_\p} N_\p=\infty, M_\p\ne 0, \id_{R_\p} M_\p<\infty, \Hdim_{R_\p} \Ext^i_{R_\p}(M_\p,N_\p)<\infty$ for any $0\le i\le\Rfd_R M\}$, so we can take a minimal element $\p$ of $S$.
Let $\q$ be a prime ideal of $R$ with $\q\subsetneq \p$.
Since $\id_{R_\p} N_\p<\infty$ and $\Hdim_{R_\p} \Ext^i_{R_\p}(M_\p,N_\p)<\infty$ for any $0\le i\le\Rfd_R M$, we have $\id_{R_\q} N_\q<\infty$ and $\Hdim_{R_\q} \Ext^i_{R_\q}(M_\q,N_\q)<\infty$ for any $0\le i\le\Rfd_R M$.
However, $\q$ is not in $S$ by the minimality of $\p$ in $S$.
Thus either $\Hid_{R_\q} N_\q<\infty$ or $M_\q=0$ holds.
Lemma \ref{fpd dualizing ext vanishing}(4) forces $\Ext^j_{R_\q}(M_\q,N_\q)=0$ for any $j>\Rfd_R M$.
Therefore $\Ext^j_{R_\p}(M_\p,N_\p)$ has finite length for any $j>\Rfd_R M$.
Theorem \ref{fid homology} concludes $\Hid_{R_\p} N_\p<\infty$.
This contradicts $\p\in S$.
It follows from the above that $\Hid_{R_\m} N_\m<\infty$ for any maximal ideal $\m$ of $R$, which induces that $\Hid_{R_\p} N_\p<\infty$ for any prime ideal $\p$ of $R$.

In the case $\Hdim=\pd$, if $R$ is finite-dimensional, then $\id_{R_\p} N_\p\le \dim R$ for any $\p\in\spec R$, which means $\id_R N<\infty$; see \cite[Theorem 3.1.17]{BH}.
Assume that $R$ admits a dualizing complex $D$ when $\Hdim=\Gdim$.
Then $\Gdim_{R_\p} (\RHom_R(D, N)_\p)<\infty$ for any $\p\in\spec R$, and hence $\Gdim_R (\RHom_R(D, N))<\infty$, which gives $\Gid_R N<\infty$.
\end{proof}

The dual versions of Corollaries \ref{pd or Gdim cor 2}, \ref{pd or Gdim cor 2-2}, and \ref{Rfd CM 1} can be derived in the same way.
These state that Question \ref{update question} (and hence Question \ref{HJM Q6.9}) has an affirmative answer.

\begin{cor}\label{id or Gid cor 2}
Let $R$ be a local ring, $M$ and $N$ nonzero finitely generated $R$-modules with $\id_R M<\infty$ and $\Hdim_R\Ext^i_R(M,N)<\infty$ for any $0\le i\le \depth R-\depth M$.
In the case $\Hdim=\Gdim$, assume that $R$ admits a dualizing complex.
Then $\Hid_R N<\infty$.
\end{cor}

\begin{proof}
As in the proof of Theorem \ref{fid homology}, $R$ is Cohen--Macaulay.
The claim follows from Lemma \ref{Rfd CM lemma} and Corollary \ref{id or Gid cor}.
\end{proof}

\begin{cor}\label{id or Gid cor 2-2}
Let $R$ be a local ring, $M$ and $N$ nonzero finitely generated $R$-modules with $\id_R M<\infty$, and $0\le r\le \depth R-\depth M$.
Suppose that $\Hdim_R\Ext^i_R(M,N)<\infty$ for any $0\le i\le r$, and that $\Ext^j_R(M,N)=0$ for any $r<j \le \depth R-\depth M$.
In the case $\Hdim=\Gdim$, assume that $R$ admits a dualizing complex.
Then $\Hid_R N<\infty$ and $r=\Rfd_R M=\depth R-\depth M$.
\end{cor}

\begin{proof}
As in the proof of Theorem \ref{fid homology}, $R$ is Cohen--Macaulay.
When $\Hdim=\pd$, we can replace $R$ with its completion, so we may assume that $R$ admits a canonical module $\omega$ in order to prove theorem.
Corollary \ref{id or Gid cor 2} shows $\Hid_R N<\infty$.
Put $s:=\pd_R(\Hom_R(\omega, M))$.
Applying $(-)\otimes_R \omega$ to a minimal free resolution $F=(0\to F_s \to \cdots \to F_0\to 0)$ of $\Hom_R(\omega, M)$, we obtain an exact sequence 
$$
0\to F_s\otimes_R \omega \to \cdots \to F_0\otimes_R \omega \to \Hom_R(\omega, M)\otimes_R \omega \cong M \to 0.
$$
It follows from the depth lemma that $\depth M\ge \depth R-s$.
If this inequality is strict, then $F_s\otimes_R \omega \to F_{s-1}\otimes_R \omega$ would split, which is a contradiction.
We get $s=\depth R-\depth M$.
As $\id_R M<\infty$ and $\Hid_R N<\infty$, we have a quasi-isomorphism
\begin{align*}
\RHom_R(\Hom_R(\omega, M), \Hom_R(\omega, N))&\simeq \RHom_R(\RHom_R(\omega, M), \RHom_R(\omega, N)) \\
&\simeq \RHom_R(\RHom_R(\omega, M)\otimes_R^\L\omega, N) \simeq \RHom_R(M, N);
\end{align*}
see \cite[Section 3.4]{C} for instance.
This implies $\Ext^j_R(\Hom_R(\omega, M), \Hom_R(\omega, N))\cong\Ext^j_R(M,N)=0$ for any $r<j \le s$.
Lemma \ref{pd G-dim keisan} forces $s=\pd_R(\Hom_R(\omega, M))=r$.
\end{proof}

There are analogous results in the graded case. 
Corollary \ref{id or Gid cor 3-2} follows immediately from Corollaries \ref{id or Gid cor 2-2} and \ref{id or Gid cor 3}.

\begin{cor}\label{id or Gid cor 3}
Let $R$ be a *local ring, $M$ and $N$ nonzero finitely generated graded $R$-modules with $\id_R M<\infty$ and $\Hdim_R\Ext^i_R(M,N)<\infty$ for any $0\le i\le \depth R-\depth M$.
In the case $\Hdim=\Gdim$, assume that $R$ admits a dualizing complex.
Then $\Hid_R N<\infty$.
\end{cor}

\begin{proof}
Let $\m$ be the *maximal ideal ideal of $R$.
Since $\m$ is a maximal ideal ideal of $R$, applying Corollary \ref{id or Gid cor 2} to $R_\m$, we obtain $\Hid_{R_\m} N_\m<\infty$, which implies $\Hid_R N<\infty$.
In fact, when $\Hid=\id$, it follows from \cite[Proposition 3.6.6]{BH} that for any prime ideal $\p$ of $R$, the $(d+1)$-th Bass number of $N$ with respect $\p$ vanishes, where $d=\dim R$. 
So $\id_R N\le d$.
We deal with the case $\Hid=\Gid$.
If $R$ admits a dualizing complex, then it is a homomorphic image of a finite-dimensional Gorenstein ring $S$; see \cite[Corollary 1.4]{Kaw}.
There is the natural surjection from $S$ to the local base ring $R_0$.
There exists a maximal ideal $\q$ of $S$ such that $R_0$ of $R$ is a homomorphic image of a Gorenstein local $S_\q$.
By this, $R$ is a homomorphic image of some graded polynomial ring over $S_\q$, which is Gorenstein *local.
Let $D$ be the graded dualizing complex obtained in this way.
Reducing to the G-dimension case via $\RHom_R(D,N)$, it follows that $\Gid_R N<\infty$ if and only if $\Gid_{R_\m} N_\m<\infty$ by an argument similar to the proof of Corollary \ref{pd or Gdim cor 3}.
\end{proof}

\begin{cor}\label{id or Gid cor 3-2}
Let $R$ be a *local ring, $M$ and $N$ nonzero finitely generated graded $R$-modules with $\id_R M<\infty$, and $0\le r\le \depth R-\depth M$.
Suppose that $\Hdim_R\Ext^i_R(M,N)<\infty$ for any $0\le i\le r$, and that $\Ext^j_R(M,N)=0$ for any $r<j \le \depth R-\depth M$.
In the case $\Hdim=\Gdim$, assume that $R$ admits a dualizing complex.
Then $\Hid_R N<\infty$ and $r=\Rfd_R M=\depth R-\depth M$.
\end{cor}

\section{Examples}

In the remainder of this paper, we present examples related to the results obtained herein.
The examples we wish to confirm the existence of finitely generated modules $M$ and $N$ such that 
\begin{itemize}
\item $\pd_R N<\infty$, $\Hdim_R M<\infty$, and $\Hdim_R(\Ext_R^i(M,N))<\infty$ for all integers $i$, or 
\item $\id_R M<\infty$, $\Hid_R N<\infty$, and $\Hdim_R(\Ext_R^i(M,N))<\infty$ for all integers $i$.
\end{itemize}
Representative examples are provided in \cite[Section 1]{HJM}, and it has been established that our results have significance.
In this section, we slightly generalize them and provide several nontrivial examples concretely.
The underlying method is as follows.

\begin{emp}\label{underlying methods}
Let $R$ be a ring and let $K$, $L$, $M$ and $N$ be finitely generated $R$-modules. 

(1) Recall that ${\rm inf}\{i \mid \Ext^i_R(M,R)\ne 0 \}$ is called the \textit{grade} of $M$ and is denoted by $\grade_R M$. The (nonzero) module $M$ is said to be \textit{$\H$-perfect of grade $g$} if $g=\grade_R M=\Hdim_R M$. If $M$ is $\H$-perfect of grade $g$, then $\Hdim_R (\Ext^g_R(M,R))=g$ and $\Ext^i_R(M,R)=0$ for $i\ne g$; see \cite[Example 1.5]{HJM} for instance. 
(It is simply called perfect when $\Hdim=\pd$.)
For an $R$-regular sequence $x_1,\ldots,x_g$, the quotient $R/(x_1,\ldots,x_g)$ is perfect of grade $g$.

(2) Suppose that there is an exact sequence $0\to K\to L\to M\to 0$ of $R$-modules.
For any $i$, we obtain an exact sequence
$$
\Ext^{i-1}_R(K,N) \to \Ext^i_R(M,N) \to \Ext^i_R(L,N) \to \Ext^i_R(K,N) \to \Ext^{i+1}_R(M,N).
$$
By this, the finiteness of the homological dimensions of $\Ext^i_R(L,N)$ is often characterized by those of $\Ext^i_R(K,N)$ and $\Ext^i_R(M,N)$.
Assume that $N$ is projective, $M$ is $\H$-perfect of grade $m$, and $K$ is $\H$-perfect of grade $k$.
If $m<k$ or $m>k+1$, then $\Ext^m_R(L,N)\cong \Ext^m_R(M,N)$, $\Ext^k_R(L,N)\cong \Ext^k_R(K,N)$, and $\Ext^i_R(L,N)=0$ for $i\notin \{m,k\}$.
If $m=k$, then $L$ is also $\H$-perfect of grade $m$.
In particular, in the two cases above, $\Hdim_R (\Ext^i_R(L,N))<\infty$ for all integer $i$.
Suppose $m=k+1$.
Then $\Ext^i_R(L,N)=0$ for $i\notin \{m,k\}$ and there is an exact sequence
$$
0 \to \Ext^k_R(L,N) \to \Ext^k_R(K,N) \xrightarrow{\phi} \Ext^m_R(M,N) \to \Ext^m_R(L,N)\to 0.
$$
Therefore, whether either $\Ext^k_R(L,N)$ or $\Ext^m_R(L,N)$ (or equivalently both) have finite homological dimension depends on the connected homomorphism $\phi$.

(3) Suppose that $R$ is a Cohen--Macaulay ring with canonical module $\omega$, and that for all integers $i$, $\Hdim_R (\Ext^i_R(M,N))<\infty$, $\operatorname{\widehat{H}-dim}_R (M)<\infty$, and $\operatorname{\overline{H}-dim}_R (N)<\infty$.
(Here, $\operatorname{\widehat{H}-dim}$ and $\operatorname{\overline{H}-dim}$ also refer to either the projective dimension or the G-dimension, and the reason for changing the notation is to emphasize that it is harmless for each of them to independently refer to either the projective dimension or the G-dimension.)
Then 
\begin{align*}
\RHom_R(M\otimes_R \omega, N\otimes_R \omega)\simeq \RHom_R(M\otimes_R^\L \omega, N\otimes_R^\L \omega)&\simeq \RHom_R(M, \RHom_R(\omega, N\otimes_R^\L \omega)) \\
&\simeq \RHom_R(M,N).
\end{align*}
We have $\operatorname{\widehat{H}id}_R (M\otimes_R \omega)<\infty$, $\operatorname{\overline{H}id}_R (N\otimes_R \omega)<\infty$, and $\Hdim_R (\Ext^i_R(M\otimes_R \omega, N\otimes_R \omega))<\infty$ for all $i$.
Conversely, if $\operatorname{\widehat{H}id}_R (M)<\infty$ and $\operatorname{\overline{H}id}_R (N)<\infty$, then there is a quasi-isomorphism
\begin{align*}
\RHom_R(\Hom_R(\omega, M), \Hom_R(\omega, N))&\simeq \RHom_R(\RHom_R(\omega, M), \RHom_R(\omega, N))\\
&\simeq \RHom_R(\RHom_R(\omega, M)\otimes_R^\L\omega, N))\\
&\simeq \RHom_R(M,N),
\end{align*}
with $\operatorname{\widehat{H}-dim}_R (\Hom_R(\omega, M))<\infty$ and $\operatorname{\overline{H}-dim}_R (\Hom_R(\omega, N))<\infty$.
Therefore, constructing examples about (Gorenstein) injective dimension is essentially no different from constructing examples about (Gorenstein) projective dimension.
Hence, in this section, we concentrate only on examples related to (Gorenstein) projective dimension.

Note that when $M$ and $N$ do not necessarily have finite projective or injective dimension, $M\otimes_R^\L \omega$ and $\RHom_R(\omega, N)$ are not even guaranteed to be bounded complexes.
This is the reason why Corollaries \ref{pd or Gdim cor 2} and \ref{id or Gid cor 2} cannot be easily deduced from \cite[Theorems 3.3 and 3.6]{HJM}.

(4) In this section, we use the well-known Horseshoe lemma several times. Since the form of the involved morphisms is also important, we briefly recall the outline of its proof here for the reader's convenience.
Suppose that $P_1 \xrightarrow{a} P_0 \xrightarrow{b} K \to 0$, $Q_1 \xrightarrow{c} Q_0 \xrightarrow{d} M \to 0$, and $0\to K \xrightarrow{e} L \xrightarrow{f} M \to 0$ are exact sequences of finitely generated $R$-modules, where $Q_0, Q_1$ are projective.
Since $Q_0$ is projective and $f$ is surjective, there is $g: Q_0\to L$ such that $fg=d$.
We have $\image (gc) \subseteq \ker f=\image e$.
Since $Q_1$ is projective and $b$ is surjective, there is $h: Q_1\to P_0$ such that $ebh=gc$. 
There is a commutative diagram 
\[
  \xymatrix@C=30pt@R=30pt
  {
    0  \ar[r]
    & P_1 \ar[r]_{\tiny \begin{pmatrix} 1 \\ 0 \end{pmatrix}} \ar[d]^a
    & P_1\oplus Q_1 \ar[r]_{\tiny \begin{pmatrix} 0& \! \! \! \! 1 \end{pmatrix}} \ar[d]^{\tiny \begin{pmatrix} a &\! \! \! \! -h \\ 0 &\! \! \! \! c \end{pmatrix}}
    & Q_1 \ar[r] \ar[d]^c
    & 0 \\
    0  \ar[r]
    & P_0 \ar[r]_{\tiny \begin{pmatrix} 1 \\ 0 \end{pmatrix}} \ar[d]^b
    & P_0\oplus Q_0 \ar[r]_{\tiny \begin{pmatrix} 0& \! \! \! \! 1 \end{pmatrix}} \ar[d]^{\tiny \begin{pmatrix} eb& \! \! \! \! g \end{pmatrix}}
    & Q_0 \ar[r] \ar[d]^d
    & 0 \\
    0  \ar[r]
    & K \ar[r]^e \ar[d]
    & L \ar[r]^f \ar[d]
    & M \ar[r] \ar[d]
    & 0 \\
    & 0
    & 0
    & 0
    &  \\
  }
\]
with exact rows and columns.

(5) We recall matrix factorizations, which are useful for constructing totally reflexive modules.
Let $S$ be a ring, $f\in S$ a nonzerodivisor on $S$, and $R=S/(f)$.
Suppose $A$ and $B$ are $n \times n$ matrices with entries in $S$ such that $AB=BA=fE$, where $E$ is the identity matrix.
As $fE$ is injective, so are $A$ and $B$.
For $z\in S^n$, $Az\in fS^n=AB S^n$ implies $z\in B S^n$, which means that the complex $(\cdots \xrightarrow{A} R^n \xrightarrow{B} R^n \xrightarrow{A} R^n \xrightarrow{B} \cdots)$ is exact.
Since the transpose matrices of $A$ and $B$, denoted ${}^t\!A$ and ${}^t\!B$, also satisfy ${}^t\!A {}^t\!B={}^t\!B {}^t\!A=fE$, one can verify that $(\cdots \xrightarrow{{}^t\!A} R^n \xrightarrow{{}^t\!B} R^n \xrightarrow{{}^t\!A} R^n \xrightarrow{{}^t\!B} \cdots)$ is exact, which means $\cok A$ and $\cok B$ are totally reflexive over $R$.
\end{emp}

Now, let us explicitly construct the examples.
The examples of quotient rings of polynomial rings appearing below can be treated in the same way even if they are quotients of formal power series rings, that is, the following discussion covers both the local case and the graded case from the previous section.

\begin{ex}\label{nontrivial example A}
Let $R$ be a ring, $x,y$ an $R$-regular sequence in $R$, and $z\in R$ a nonzerodivisor on $R$.
There is a short exact sequence $0\to R/(x,y) \xrightarrow{z} R/(xz,yz) \to R/(z) \to 0$.
Since $R/(x,y)$ is perfect of grade 2 and $R/(z)$ is perfect of grade 1, we have 
$$
\Ext^i_R(R/(xz,yz), R)=
\begin{cases}
\Ext^i_R(R/(z), R)\cong R/(z) & (i=1), \\
\Ext^i_R(R/(x,y), R)\cong R/(x,y) & (i=2), \\
0 & ({\rm otherwise}).
\end{cases}
$$
We see that $0 \to R \xrightarrow{\tiny \begin{pmatrix} -y \\ x \end{pmatrix}} R^2 \xrightarrow{\tiny \begin{pmatrix} xz &\! \! \! \! yz \end{pmatrix}} R \to 0$ is a free resolution of $R/(xz,yz)$; see \cite[Theorem 1.4.13]{BH} for instance.
In particular, $\pd_R (R/(xz,yz))<\infty$ and $\pd_R(\Ext_R^i(R/(xz,yz),R))<\infty$ for all integers $i$.
We now develop this one step further.

For any $a\in R$, there exists a pushout diagram
\[
  \xymatrix@C=30pt@R=20pt
  {
    0  \ar[r]
    & R \ar[r]^{\phi:=\tiny \begin{pmatrix} xz \\ xy \end{pmatrix}} \ar[d]_\alpha
    & R^2 \ar[r] \ar[d]^\beta
    & C \ar[r] \ar@{=}[d]
    & 0 \\
    0  \ar[r]
    & A:=R/(xz,yz) \ar[r]^{\qquad \gamma}
    & B \ar[r]^\delta
    \ar@{}[ul]|{\mathrm{PO}}
    & C \ar[r] 
    & 0,
  }
\]
where $\alpha(r)=ar+(xz,yz)$ and $C:=\cok \phi$.
Note that $\phi$ is injective as $x$ and $z$ are nonzerodivisors on $R$.
We get $\Hom_R(C,R)=\ker (\Hom_R(\phi,R))\cong R$ and $\Ext^1_R(C,R)=\cok (\Hom_R(\phi,R))\cong A$ by the free resolution of $A$.
(This construction of $C$ is the same as that in \cite[Example 1.6]{HJM}.)
Then
\begin{itemize}
\item $\Hom_R(B, R)\cong \Hom_R(C,R)\cong R$, 
\item there is an exact sequence $0\to \Ext^1_R(C, R)\cong A \to \Ext^1_R(B, R) \to \Ext^1_R(A, R)\cong R/(z)\to 0$,
\item $\Ext^2_R(B, R) \cong \Ext^2_R(A, R)\cong R/(x,y)$,
\item $\Ext^i_R(B, R)=0$ for $i\ge 3$.
\end{itemize}
Therefore, $B$ and $\Ext^i_R(B, R)$ have finite projective dimension for any integer $i$.
Applying \ref{underlying methods}(4), we can explicitly give a presentation of $B$, which is
$$
0 \to R \xrightarrow{\tiny \begin{pmatrix} -y \\ x\\ 0 \end{pmatrix}} R^3 \xrightarrow{\tiny \begin{pmatrix} xz &\! \! \! \! yz & \! \! \! \! a \\ 0 &\! \! \! \! 0 & \! \! \! \! xz \\ 0 &\! \! \! \! 0 & \! \! \! \! yz \end{pmatrix}} R^3 \xrightarrow{\beta} B \to 0.
$$

Let $R=k[x,y,z,v,w]/(xy,zw,zv)$ be a quotient of a polynomial ring over a field $k$.
Then $R$ is of dimension 3 and is not Cohen--Macaulay. 
Now $x+y, z+w$ is an $R$-regular sequence and $z+v$ is a nonzerodivisor on $R$.
By the previous paragraph, the cokernel of the following homomorphism satisfies the hypotheses of Corollary \ref{pd or Gdim cor 3}, where $\Hdim=\pd$ and $N=R$:
$$
R^3 \xrightarrow{\tiny \begin{pmatrix} xz+yz+xv+yv &\! \! \! \! z^2+vw & \! \! \! \! v^2+w^2 \\ 0 &\! \! \! \! 0 & \! \! \! \! xz+yz+xv+yv \\ 0 &\! \! \! \! 0 & \! \! \! \! z^2+vw \end{pmatrix}} R^3.
$$
Of course, the top-right (homogeneous) entry of the matrix can be arbitrary.
\end{ex}

\begin{ex}
Let $R$ be a ring, and $x,y$ an $R$-regular sequence in $R$.
Choose nonzerodivisors $a, b$ in $(x,y)R$ and elements $c,d\in R$. 
We write $a=a_1 x+a_2 y$ and $b=b_1 x+b_2 y$ for some $a_1,a_2,b_1,b_2\in R$.
Then there exists a commutative diagram 
\[
  \xymatrix@C=80pt@R=30pt
  {
    &
    & 0 \ar[d]
    & 0 \ar[d]
    &  \\
    & 0 \ar[r] \ar[d]
    & R \ar@{=}[r] \ar[d]_{\tiny \begin{pmatrix} c \\ d \\ -y \\ x\end{pmatrix}}
    & R \ar[r] \ar[d]^{\tiny \begin{pmatrix} -y\\ x \end{pmatrix}} 
    & 0 \\
    0  \ar[r]
    & R^2 \ar[r]^{\tiny \begin{pmatrix} 1& \! \! \! \! 0 \\ 0& \! \! \! \! 1 \\ 0& \! \! \! \! 0 \\ 0& \! \! \! \! 0\end{pmatrix}} \ar[d]_{\tiny \begin{pmatrix} a &\! \! \! \! 0 \\ 0 &\! \! \! \! b \end{pmatrix}}
    & R^4 \ar[r]^{\tiny \begin{pmatrix} 0& \! \! \! \! 0& \! \! \! \! 1& \! \! \! \! 0 \\ 0& \! \! \! \! 0& \! \! \! \! 0& \! \! \! \! 1\end{pmatrix}} \ar[d]^{\psi:=\tiny \begin{pmatrix} a& \! \! \! \! 0& \! \! \! \! ca_2& \! \! \! \! -ca_1 \\ 0& \! \! \! \! b& \! \! \! \! db_2& \! \! \! \! -db_1 \\ 0& \! \! \! \! 0& \! \! \! \! x& \! \! \! \! y\end{pmatrix}}
    & R^2 \ar[r] \ar[d]^{\tiny \begin{pmatrix} x& \! \! \! \! y \end{pmatrix}} 
    & 0 \\
    0  \ar[r]
    & R^2 \ar[r]^{\tiny \begin{pmatrix} 1& \! \! \! \! 0 \\ 0& \! \! \! \! 1 \\ 0& \! \! \! \! 0 \end{pmatrix}}
    & R^3 \ar[r]^{\tiny \begin{pmatrix} 0& \! \! \! \! 0& \! \! \! \! 1 \end{pmatrix}} 
    & R \ar[r] 
    & 0 \\
  }
\]
with exact rows and columns; see \cite[Theorem 1.4.13]{BH} for instance.
The exact sequence 
$$0\to R/(a)\oplus R/(b) \to M\to R/(x,y)\to 0$$
is induced by the Snake Lemma, where $M=\cok\psi$.
Note that $R/(a)\oplus R/(b)$ is perfect of grade 1 and $R/(x,y)$ is perfect of grade 2.
Regarding the exact sequence appearing in \ref{underlying methods}(2), we obtain the following commutative diagram with exact rows by computing the connected homomorphism:
 \[
  \xymatrix@C=20pt@R=30pt
  {
    0  \ar[r]
    & \Ext^1_R(M,R) \ar[r]
    & \Ext^1_R(R/(a)\oplus R/(b),R) \ar[r] \ar[d]^\cong
    & \Ext^2_R(R/(x,y),R) \ar[r] \ar[d]^\cong
    & \Ext^2_R(M,R) \ar[r]
    & 0 \\
    & 
    & R/(a)\oplus R/(b) \ar[r]^\phi
    & R/(x,y) \ar[r]
    & R/(x,y,c,d) \ar[r] 
    & 0, \\
  }
\]
where $\phi(r+(a), s+(b))=cr+ds+(x,y,c,d)$ for $r,s\in R$.
Hence $\Hdim_R (R/(x,y,c,d))<\infty$ if and only if $\Hdim_R (\Ext_R^i(M,R))<\infty$ for every $i$.

Let us consider the quotient $R=k[x,y,z,v,w]/(xy,zw,zv)$ of a polynomial ring over a field $k$ and an $R$-regular sequence $x+y, z+w$.
Put $S=R/(x+y, z+w)R\cong k[x,z,v]/(x^2,z^2,zv)$ and $X=S/xS\cong R/(x+y, z+w, x)R\cong k[z,v]/(z^2,zv)$.
It is seen that $(\cdots \xrightarrow{x} S \xrightarrow{x} S \xrightarrow{x} S \to X\to 0)$ is exact, and thus $\Gdim_S X=0$; see \ref{underlying methods}(5).
As $x+y, z+w$ is an $R$-regular sequence, we get $\Gdim_R X=\Gdim_S X+2=2$ by \cite[Proposition 1.5.3]{C}.
By the previous paragraph, the cokernel $M$ of the following homomorphism satisfies the hypotheses of Corollary \ref{pd or Gdim cor 3}, where $\Hdim=\Gdim$ and $N=R$:
$$
R^4 \xrightarrow{\tiny \begin{pmatrix} x+y& \! \! \! \! 0& \! \! \! \! 0& \! \! \! \! -x \\ 0& \! \! \! \! z+w& \! \! \! \! x& \! \! \! \! 0 \\ 0& \! \! \! \! 0& \! \! \! \! x+y& \! \! \! \! z+w\end{pmatrix}} R^3 \to M \to 0.
$$
Note that the $M$ has finite projective dimension, and thus so does $\RHom_R(M,R)$,  but its second Ext module $X=\Ext^2_R(M,N)$ satisfies $\pd_R X=\pd_S X+2=\infty$ since $x+y, z+w\notin (x,y,z,v,w)^2R$.
\end{ex}

We provide an example of a module that has infinite projective dimension but finite G-dimension.

\begin{ex}
Let $S=k[x,y,z,v,w]/(v^2, vx+vw, xy+z^2)$ be a quotient of a polynomial ring over a field $k$ and $R=S/(y^2-wz)S$.
As a remark, since $k[x,y,z]/(xy+z^2)\cong k[s^2, st, t^2]\subseteq k[s,t]$, $xy+z^2$ is a prime element.
Note that 
$$
(v^2, vx+vw)=(v^2,x+w)\cap (v) \ {\rm and} \ (v^2, vx+vw, xy+z^2)=(v, xy+z^2)\cap (v^2, x+w, xy+z^2)
$$ 
are the primary decompositions of $(v^2, vx+vw)$ and $(v^2, vx+vw, xy+z^2)$, respectively.
Thus $xy+z^2$ is a nonzerodivisor on $k[x,y,z,v,w]/(v^2, vx+vw)$, and $y^2-wz$, $z$, and $w$ are nonzerodivisors on $S$.
On the other hand, $w$ is a nonzerodivisor on $R$.
Indeed, we see that $y^2-wz$ is a nonzerodivisor on $S/(w)S\cong k[x,y,z,v]/(v^2, vx, xy+z^2)$, which means that $w, y^2-wz$ is an $S$-regular sequence.
Viewing $S$ as a standard graded ring with a unique maximal ideal, we see that the reordered sequence $y^2-wz, w$ is also an $S$-regular sequence.
To summarize the above discussion, it is seen that $R$ is a non-Cohen--Macaulay ring of dimension 2.

By \ref{underlying methods}(5), we have exact sequences
$$
\cdots S^2 \xrightarrow{\tiny \begin{pmatrix} x& \! \! \! \! -z \\ z& \! \! \! \! y \end{pmatrix}} S^2 \xrightarrow{\tiny \begin{pmatrix} y& \! \! \! \! z \\ -z& \! \! \! \! x \end{pmatrix}} S^2 \xrightarrow{\tiny \begin{pmatrix} x& \! \! \! \! -z \\ z& \! \! \! \! y \end{pmatrix}} S^2 \cdots \ {\rm and } \ 
\cdots R^2 \xrightarrow{\tiny \begin{pmatrix} w& \! \! \! \! y \\ y& \! \! \! \! z \end{pmatrix}} R^2 \xrightarrow{\tiny \begin{pmatrix} z& \! \! \! \! -y \\ -y& \! \! \! \! w \end{pmatrix}} R^2 \xrightarrow{\tiny \begin{pmatrix} w& \! \! \! \! y \\ y& \! \! \! \! z \end{pmatrix}} R^2 \cdots.
$$
We consider the following exact sequences of complexes of free modules:
\[
  \xymatrix@C=30pt@R=30pt
  {
    & 0 \ar[r] \ar[d]
    & S \ar[r]^{\tiny \begin{pmatrix} y \\ -z \end{pmatrix}} \ar[d]^{\tiny \begin{pmatrix} 1 \\ 0\end{pmatrix}}
    & S^2 \ar@{=}[d] 
    &
    & 0 \ar[r] \ar[d]
    & R \ar[r]^{\tiny \begin{pmatrix} w \\ y \end{pmatrix}} \ar[d]^{\tiny \begin{pmatrix} 1 \\ 0\end{pmatrix}}
    & R^2 \ar@{=}[d] \\
    S^2  \ar[r]^{\tiny \begin{pmatrix} y& \! \! \! \! z \\ -z& \! \! \! \! x \end{pmatrix}} \ar@{=}[d]
    & S^2 \ar[r]^{\tiny \begin{pmatrix} x& \! \! \! \! -z \\ z& \! \! \! \! y \end{pmatrix}} \ar@{=}[d]
    & S^2 \ar[r]_{\tiny \begin{pmatrix} y& \! \! \! \! z \\ -z& \! \! \! \! x \end{pmatrix}} \ar[d]_{\tiny \begin{pmatrix} 0& \! \! \! \! 1 \end{pmatrix}} 
    & S^2 \ar[d] 
    & R^2  \ar[r]^{\tiny \begin{pmatrix} w& \! \! \! \! y \\ y& \! \! \! \! z \end{pmatrix}} \ar@{=}[d]
    & R^2 \ar[r]^{\tiny \begin{pmatrix} z& \! \! \! \! -y \\ -y& \! \! \! \! w \end{pmatrix}} \ar@{=}[d]
    & R^2 \ar[r]_{\tiny \begin{pmatrix} w& \! \! \! \! y \\ y& \! \! \! \! z \end{pmatrix}} \ar[d]_{\tiny \begin{pmatrix} 0& \! \! \! \! 1 \end{pmatrix}} 
    & R^2 \ar[d] \\
    S^2 \ar[r]^{\tiny \begin{pmatrix} y& \! \! \! \! z \\ -z& \! \! \! \! x \end{pmatrix}}
    & S^2 \ar[r]^{\tiny \begin{pmatrix} z& \! \! \! \! y \end{pmatrix}}
    & S \ar[r]
    & 0  
    & R^2 \ar[r]^{\tiny \begin{pmatrix} w& \! \! \! \! y \\ y& \! \! \! \! z \end{pmatrix}}
    & R^2 \ar[r]^{\tiny \begin{pmatrix} -y& \! \! \! \! w \end{pmatrix}}
    & R \ar[r]
    & 0.  \\
  }
\]
The middle rows and both $0\to S \xrightarrow{\tiny \begin{pmatrix} y \\ -z \end{pmatrix}} S^2$ and $0\to R \xrightarrow{\tiny \begin{pmatrix} w \\ y \end{pmatrix}} R^2$ are all exact as $z$ is a nonzerodivisor on $S$ and $w$ is a nonzerodivisor on $R$.
The induced long exact sequences of homologies show that 
$$
S^2 \xrightarrow{\tiny \begin{pmatrix} y& \! \! \! \! z \\ -z& \! \! \! \! x \end{pmatrix}} S^2 \xrightarrow{\tiny \begin{pmatrix} z& \! \! \! \! y \end{pmatrix}} S \ {\rm and}\ 
R^2 \xrightarrow{\tiny \begin{pmatrix} w& \! \! \! \! y \\ y& \! \! \! \! z \end{pmatrix}} R^2 \xrightarrow{\tiny \begin{pmatrix} -y& \! \! \! \! w \end{pmatrix}} R
$$
are exact.
By this, the $S$-module $S/(y, z)S$ is $\operatorname{G}$-perfect of grade 1 and the $R$-module $R/(y, w)R$ is $\operatorname{G}$-perfect of grade 1.
It follows from \cite[Proposition 1.5.3]{C} that $R/(y, z)R\cong S/(y, z)S$ is totally reflexive as an $R$-module.
Hereafter, based on the method in \ref{underlying methods}(2), we construct a nontrivial short exact sequence $0\to R/(y, w)R \to M \to R/(y, z)R \to 0$ of finitely generated $R$-modules having finite G-dimension.

All such short exact sequences appear as pushout diagrams 
\[
  \xymatrix@C=30pt@R=20pt
  {
    0  \ar[r]
    & (y,z)R \ar[r]^{\iota} \ar[d]^f
    & R \ar[r]^{\pi \qquad} \ar[d]
    & R/(y,z)R \ar[r] \ar@{=}[d]
    & 0 \\
    0  \ar[r]
    & R/(y,w)R \ar[r]
    & M \ar[r]
    \ar@{}[ul]|{\mathrm{PO}}
    & R/(y,z)R \ar[r] 
    & 0,
  }
\]
associated with some $R$-homomorphism $f\in\Hom_R((y,z)R, R/(y,w)R)$, where $\iota$ and $\pi$ are the natural injection and surjection, respectively.
Considering the long exact sequence
$$
\Hom_R(R, R/(y,w)R)\xrightarrow{\Hom_R(\iota, R/(y, w)R)} \Hom_R((y,z)R, R/(y,w)R)\xrightarrow{\xi} \Ext^1_R(R/(y,z)R, R/(y,w)R),
$$
the induced short exact sequence $0\to R/(y, w)R \to M \to R/(y, z)R \to 0$ does not split if and only if $f\notin\ker \xi=\image\Hom_R(\iota, R/(y, w)R)$.
We prove that there exists a homomorphism $f:(y,z)R\to R/(y,w)R$ with $f(ya+zb)=vza+(y,w)R$ for $a,b\in R$.
Put $\m=(x,y,z,v,w)R$.
The $R$-homomorphism $R \to R/(y,w)R: a \mapsto vza+(y,w)R$ induces an $R$-homomorphism $F: R/\m \to R/(y,w)R: a+\m \mapsto vza+(y,w)R$.
On the other hand, the natural surjection $R^2 \to (y,z)R: {}^t\!\begin{pmatrix} a& \! \! \! b\end{pmatrix}\mapsto ya+zb$ induces an isomorphism $G: (R/\m)^2 \xrightarrow{\cong} (y,z)R/\m (y,z)R$.
The desired homomorphism $f$ is the composition 
$$(y,z)R \twoheadrightarrow (y,z)R/\m (y,z)R \xrightarrow{G^{-1}} (R/\m)^2 \xrightarrow{\tiny \begin{pmatrix} 1& \! \! \! \! 0\end{pmatrix}} R/\m \xrightarrow{F} R/(y,w)R.$$
We see that $f\notin\image\Hom_R(\iota, R/(y, w)R)$ since $f(yR)\ne 0$ in $R/(y, w)R$.

Applying the Horseshoe lemma to the short exact sequence  $0\to R/(y, w)R \to M \to R/(y, z)R \to 0$ obtained from the pushout diagram concerning the homomorphism $f$ in the previous paragraph, we get the commutative diagram 
\[
  \xymatrix@C=30pt@R=20pt
  {
    0  \ar[r]
    & R^2 \ar[r]^{\tiny \begin{pmatrix} 1& \! \! \! \! 0 \\ 0& \! \! \! \! 1 \\ 0& \! \! \! \! 0 \\ 0& \! \! \! \! 0\end{pmatrix}} \ar[d]^{\tiny \begin{pmatrix} y& \! \! \! \! w \end{pmatrix}}
    & R^4 \ar[r]^{\tiny \begin{pmatrix}  0& \! \! \! \! 0& \! \! \! \! 1& \! \! \! \! 0 \\ 0& \! \! \! \! 0& \! \! \! \! 0& \! \! \! \! 1 \end{pmatrix}} \ar[d]^{\tiny \begin{pmatrix}  y& \! \! \! \! w& \! \! \! \! -vz& \! \! \! \! 0 \\ 0& \! \! \! \! 0& \! \! \! \! y& \! \! \! \! z \end{pmatrix}}
    & R^2 \ar[r] \ar[d]^{\tiny \begin{pmatrix} y& \! \! \! \! z \end{pmatrix}}
    & 0 \\
    0  \ar[r]
    & R \ar[r]_{\tiny \begin{pmatrix} 1 \\ 0 \end{pmatrix}} \ar[d]
    & R^2 \ar[r]_{\tiny \begin{pmatrix} 0& \! \! \! \! 1 \end{pmatrix}} \ar[d]
    & R \ar[r] \ar[d]
    & 0 \\
    0  \ar[r]
    & R/(y, w)R \ar[r] \ar[d]
    & M \ar[r] \ar[d]
    & R/(y, z)R \ar[r] \ar[d]
    & 0 \\
    & 0
    & 0
    & 0
    &  \\
  }
\]
with exact rows and columns; see \ref{underlying methods}(4).
With this, we obtain an explicit presentation of $M$.
By \ref{underlying methods}(2), 
$$
\Ext^i_R(M, R)=
\begin{cases}
\Hom_R(R/(y, z), R) & (i=0), \\
\Ext^1_R(R/(y, w), R) & (i=1), \\
0 & ({\rm otherwise}).
\end{cases}
$$
The short exact sequence $0\to R\xrightarrow{w} R\to R/(w)R\to 0$ induces a long exact sequence
\begin{align*}
0&\to \Hom_R(M, R)\xrightarrow{w} \Hom_R(M, R) \to \Hom_R(M, R/(w)R) \\
&\to \Ext^1_R(M, R)\xrightarrow{w} \Ext^1_R(M, R) \to \Ext^1_R(M, R/(w)R) \to \Ext^2_R(M, R)=0.
\end{align*} 
The multiplication by $w$ on $\Ext^1_R(M, R)$ is the zero map since $\Ext^1_R(M, R)\cong \Ext^1_R(R/(y, w), R)$.
Then
\begin{itemize}
\item $0\to \Hom_R(M, R)\xrightarrow{w} \Hom_R(M, R) \to \Hom_R(M, R/(w)R) \to \Ext^1_R(M, R)\to 0$ is exact,
\item $\Ext^1_R(M, R/(w)R) \cong \Ext^1_R(M, R)$,
\item $\Ext^i_R(M, R/(w)R)=0$ for $i\ge 2$.
\end{itemize}
As $\Hom_R(M, R)\cong\Hom_R(R/(y, z), R)$ and $\Ext^1_R(M, R)\cong \Ext^1_R(R/(y, w), R)$ have finite G-dimension, so do $\Hom_R(M, R/(w)R)$ and $\Ext^1_R(M, R/(w)R)$.
In particular, $M$ and $N=R/(w)R$ satisfies the condition stated in the third line at the beginning of this section. (Note that $R$ is graded, but $M$ is not.)

Finally, we note that the $R$-module $M$ has infinite projective dimension.
Indeed, if $M$ admits a bounded projective resolution over $R$, then it is a resolution by modules having finite projective dimension over $S$ since $\pd_S R=1$, which means $\pd_S M<\infty$.
As seen above, $w$ is a nonzerodivisor on $S$, and it is easy to see that $y$ is a nonzerodivisor on $S/(w)S\cong k[x,y,z,v]/(v^2, vx, xy+z^2)$.
Hence $w, y$ is an $S$-regular sequence, and thus $S/(y, w)S\cong R/(y, w)R$ has finite projective dimension over $S$.
The exact sequence $0\to R/(y, w)R \to M \to R/(y, z)R \to 0$ yields that $S/(y, z)S\cong R/(y, z)R$ has finite projective dimension over $S$.
However, the complex
$$
\cdots S^2 \xrightarrow{\tiny \begin{pmatrix} x& \! \! \! \! -z \\ z& \! \! \! \! y \end{pmatrix}} S^2 \xrightarrow{\tiny \begin{pmatrix} y& \! \! \! \! z \\ -z& \! \! \! \! x \end{pmatrix}} S^2 \xrightarrow{\tiny \begin{pmatrix} x& \! \! \! \! -z \\ z& \! \! \! \! y \end{pmatrix}} S^2 \xrightarrow{\tiny \begin{pmatrix} y& \! \! \! \! z \\ -z& \! \! \! \! x \end{pmatrix}} S^2 \xrightarrow{\tiny \begin{pmatrix} z& \! \! \! \! y \end{pmatrix}} S \to 0
$$
is a (minimal) free resolution of $S/(y, z)S$, this leads to a contradiction.
\end{ex}


\subsection*{Complete intersection dimension} 

To close, we consider complete intersection dimension briefly.

\begin{dfn}
Let $R$ be a local ring.
A \textit{quasi-deformation} of $R$ is a pair of homomorphism $R\to R' \leftarrow S$ of local rings, where $R\to R'$ is flat and $R' \leftarrow S$ is surjective with the kernel generated by a regular sequence. 
For finitely generated $R$-module $M$, the \textit{complete intersection dimension} of $M$ is defined to be $\CIdim_R M:={\rm inf}\{\pd_S (M\otimes_R R')-\pd_S R' \mid R\to R' \leftarrow S$ is a quasi-deformation of $R\}$.
\end{dfn}

For finitely generated module $M$ over a local ring $R$, it has finite complete intersection dimension if and only if there is a quasi-deformation $R\to R' \leftarrow S$ of $R$ such that $\pd_S (M\otimes_R R')<\infty$.
It is a difficult problem whether, when two distinct modules $M$ and $N$ both have finite complete intersection dimension, one can choose a common quasi-deformation $R\to R' \leftarrow S$ of $R$ such that $\pd_S (M\otimes_R R')<\infty$ and $\pd_S (N\otimes_R R')<\infty$.
Apart from that difficulty, we obtain the complete intersection dimension version of Corollaries \ref{pd or Gdim cor 2}, \ref{pd or Gdim cor 2-2}, and \ref{Rfd CM 1}.
All similar results follow in the same way from the following proposition.

\begin{prop}\label{C.I-dim cor}
Let $R$ be a local ring, and $M$ and $N$ nonzero finitely generated $R$-modules with $\pd_R N<\infty$.
If there is a quasi-deformation $R\to R' \leftarrow S$ of $R$ such that $\pd_S (\Ext^i_R(M,N)\otimes_R R')<\infty$ for any $0\le i\le \Rfd_R M$.
Then $\pd_S (M\otimes_R R')<\infty$, in particular, $\CIdim_R M<\infty$.
\end{prop}

We prepare one lemma and prove Proposition \ref{C.I-dim cor}.

\begin{lem}\label{C.I-dim lemma}
Let $(R,\m)$ be a local ring, $x\in\m$ a nonzerodivisor on $R$, and $M$ and $N$ nonzero finitely generated $R/xR$-modules.
Suppose that $S$ is a ring and $R$ is an $S$-algebra.
If $\Hdim_S(\RHom_{R/xR}(M,N))<\infty$, then $\Hdim_S(\RHom_{R}(M,N))<\infty$.
\end{lem}

\begin{proof}
Let $F^{R/xR}_M$ be an $(R/xR)$-free resolution of $M$ and $F^R_M$ an $R$-free resolution of $M$.
Similarly to the proof of \cite[Lemma 4.2]{KLSL}, there is a short exact sequence 
$$
0\to F^{R/xR}_M[-1] \to  R/xR\otimes_R F^{R}_M \to  F^{R/xR}_M\to 0
$$
of complexes of free $(R/xR)$-modules.
This indeces a short exact sequence 
$$
0\to \Hom_{R/xR}(F^{R/xR}_M, N) \to \Hom_{R/xR}(R/xR\otimes_RF^{R}_M,N) \to \Hom_{R/xR}(F^{R/xR}_M,N)[1] \to 0.
$$
From this, the claim holds since $\RHom_{R/xR}(M,N)=\Hom_{R/xR}(F^{R/xR}_M, N)$ and $\RHom_{R}(M,N)=\Hom_{R}(F^{R}_M,N)\cong \Hom_{R/xR}(R/xR\otimes_RF^{R}_M,N)$
\end{proof}

\begin{proof}[Proof of Proposition \ref{C.I-dim cor}]
Put $r=\Rfd_R M$.
Suppose that there is a quasi-deformation $R\to R' \leftarrow S$ such that $\pd_S(\Ext^i_{R'}(M',N'))<\infty$ for any $0\le i\le r$, where $(-)'=(-)\otimes_R R'$.
We prove $\pd_S M'<\infty$.

For any $0\le i\le r$, $\Ext^i_R(M,N)$ has finite G-dimension over $R$ since $\pd_S(\Ext^i_{R'}(M',N'))<\infty$.
By Corollary \ref{pd or Gdim cor 2-2}, we have $\Gdim_R M=r$.
Lemma \ref{pd G-dim keisan}(2) asserts that $\Ext^i_R(M,N)=0$ for all $i>r$.
Then $\Ext^i_{R'}(M',N')=0$ for all $i>r$ as $R'$ is flat over $R$, which means $\pd_S (\RHom_{R'}(M',N'))<\infty$.
Lemma \ref{C.I-dim lemma} implies $\pd_S (\RHom_{S}(M',N'))<\infty$.
As $\Gdim_{R'} M'<\infty$ and $\pd_{R'} N'<\infty$, we obtain $\Gdim_S M'<\infty$ and $\pd_S N'<\infty$.
Let $\widehat{S}$ be the completion of $S$ and $D$ the dualizing complex of $\widehat{S}$.
Since $\Gdim_{\widehat{S}} (M'\otimes_S \widehat{S})<\infty$ and $\pd_{\widehat{S}} (N'\otimes_S \widehat{S})<\infty$, $(M'\otimes_S \widehat{S})\otimes_{\widehat{S}}^\L D$ is quasi-isomorphic to a bounded complex of finitely generated $\widehat{S}$-modules and we get
\begin{align*}
\RHom_{\widehat{S}}(M'\otimes_S \widehat{S},N'\otimes_S \widehat{S})&\simeq
\RHom_{\widehat{S}}(M'\otimes_S \widehat{S},\RHom_{\widehat{S}}(D,  (N'\otimes_S \widehat{S})\otimes_{\widehat{S}}^\L D)) \\
&\simeq \RHom_{\widehat{S}}((M'\otimes_S \widehat{S})\otimes_{\widehat{S}}^\L D, (N'\otimes_S \widehat{S})\otimes_{\widehat{S}}^\L D).
\end{align*}
Note that $\pd_{\widehat{S}} (\RHom_{\widehat{S}}(M'\otimes_S \widehat{S},N'\otimes_S \widehat{S}))<\infty$ and $\id_{\widehat{S}} ((N'\otimes_S \widehat{S})\otimes_{\widehat{S}}^\L D)<\infty$.
It follows from \cite[Lemma 6.2.12]{CF} that $\id_{\widehat{S}} ((M'\otimes_S \widehat{S})\otimes_{\widehat{S}}^\L D)<\infty$, and hence $\pd_{\widehat{S}} (M'\otimes_S \widehat{S})<\infty$. So $\pd_S M'<\infty$.
\end{proof}

Below is a direct corollary of Proposition \ref{C.I-dim cor}.

\begin{cor}\label{C.I-dim cor 2}
Let $R$ be a local ring, and $M$ and $N$ nonzero finitely generated $R$-modules with $\pd_R N<\infty$.
Suppose that there is $r\ge 0$ such that $\CIdim_R (\Ext^r_R(M,N))<\infty$ and $\Ext^i_R(M,N)=0$ for any $0\le i\le \Rfd_R M$ with $i\ne r$.
Then $\CIdim_R M=r=\Rfd_R M$.
\end{cor}

\begin{ac}
The author would like to thank his supervisor Ryo Takahashi for valuable suggestions and comments.
The author thanks Yuki Mifune for pointing out where a fact used in this paper is stated explicitly, Victor Daniel Mendoza Rubio for pointing out several errors in a previous version of this paper, and Yuya Otake for doing both.
The author was partly supported by Grant-in-Aid for JSPS Fellows Grant Number 23KJ1117.
\end{ac}


\end{document}